\documentclass[final,onefignum,onetabnum]{siamart190516}

\newtheorem{example}{Example}
\usepackage{lipsum}
\usepackage{amsfonts}
\usepackage{graphicx}
\usepackage{epstopdf}
\usepackage{algorithmic}
\usepackage{arydshln}

\ifpdf
  \DeclareGraphicsExtensions{.eps,.pdf,.png,.jpg}
\else
  \DeclareGraphicsExtensions{.eps}
\fi

\newcommand{\define}{\stackrel{\text{\tiny def}}{=}}

\newsiamremark{remark}{Remark}
\newsiamremark{hypothesis}{Hypothesis}
\crefname{hypothesis}{Hypothesis}{Hypotheses}
\newsiamthm{claim}{Claim}

\headers{SHERMAN-MORRISON-WOODBURY IDENTITY FOR TENSORS}{Shih Yu Chang}

\title{SHERMAN-MORRISON-WOODBURY IDENTITY FOR TENSORS}

\author{Shih Yu Chang\thanks{Department of Applied Data Science, San Jose State University, San Jose, CA 95192, USA 
  (\email{ shihyu.chang@sjsu.edu } ).}
}

\usepackage{amsopn}

\ifpdf
\hypersetup{
  pdftitle={SHERMAN-MORRISON-WOODBURY IDENTITY FOR TENSORS},
  pdfauthor={Shih Yu Chang}
}
\fi

\begin{document}

\maketitle

\begin{abstract}
In linear algebra, the Sherman–Morrison–Woodbury identity says that the inverse of a rank-$k$ correction of some matrix can be computed by doing a rank-k correction to the inverse of the original matrix. This identity is crucial to accelerate the matrix inverse computation when the matrix involves correction. Many scientific and engineering applications have to deal with this matrix inverse problem after updating the matrix, e.g., sensitivity analysis of linear systems, covariance matrix update in Kalman filter, etc. However, there is no similar identity in tensors. In this work, we will derive the Sherman–Morrison–Woodbury identity for invertible tensors first. Since not all tensors are invertible, we further generalize the Sherman–Morrison–Woodbury identity for tensors with Moore-Penrose generalized inverse by utilizing orthogonal projection of the correction tensor part into the original tensor and its Hermitian tensor. According to this new established the Sherman–Morrison–Woodbury identity for tensors, we can perform sensitivity analysis for multilinear systems by deriving the normalized upper bound for the solution of a multilinear system. Several numerical examples are also presented to demonstrate how the normalized error upper bounds are affected by perturbation degree of tensor coefficients.  
\end{abstract}



\section{Introduction}



Tensors are higher-order generalizations of matrices and vectors, which have been studied abroadly due to
the practical applications in many scientific and engineering fields~\cite{kolda2009tensor, qi2017tensor, wei2016theory}, including psychometrics~\cite{tucker1966some}, digital image restorations~\cite{qi2010higher}, quantum
entanglement~\cite{hu2012geometric, qi2012minimum}, signal processing~\cite{de2007fourth, ortiz2019sparse, zhou2019tensor, kanatsoulis2019regular},  high-order statistics~\cite{cardoso1999high, ding2018fast}, automatic control~\cite{precup2012novel}, spectral hypergraph theory~\cite{hu2015laplacian, cooper2020adjacency, sawilla2008survey}, higher order Markov chains~\cite{li2019c, liu2019relaxation,Kwak2015High}, magnetic resonance
imaging~\cite{qi2010higher, qi2008d}, algebraic geometry~\cite{cartwright2013number, li2012characteristic}, Finsler geometry~\cite{balan2010applications}, image authenticity verification~\cite{zhang2013gradient}, and so on. More applications about tensors can be found at~\cite{kolda2009tensor, qi2017tensor}


In tensor data analysis, the data sets are represented by tensors, while the associated multilinear algebra problems can be formulated for various data-processing tasks, such as web-link analysis~\cite{kolda2005higher, kolda2006tophits}, document analysis~\cite{cai2006tensor, liu2005text}, information retrieval~\cite{ng2011multirank, li2012har}, model learning~\cite{cheng2020learning, brinton2016mining}, data-model reduction~\cite{shin2016fully, cheng2018scaling, phan2018error}, model prediction~\cite{chen2019incremental}, movie recommendation~\cite{tang2012cross}, and videos analysis~\cite{li2011general, tao2007general}, numerical PDE~\cite{sahoo2020reverse}. Most of the above-stated tensor-formulated methodologies depend on the solution to the following \emph{tensor equation} (a.k.a. \emph{multilinear system of equations}~\cite{kolda2005higher, OVMMUPH2011}): 
\begin{eqnarray}\label{eq: tensor eq for introduction}
\mathcal{A}\star_N\mathcal{X} &=& \mathcal{B},
\end{eqnarray}
where $\mathcal{A}, \mathcal{B}$ are tensors and $\star_N$ denotes the \emph{Einstein product}with order $N$~\cite{sun2018generalized}. Basically, there are two main approaches to solve the unknown tensor $\mathcal{X}$. The first approach is to solve the Eq.~\eqref{eq: tensor eq for introduction} iteratively. Three primary iterative algorithms are \emph{Jacobi method}, \emph{Gauss-Seidel method}, and \emph{Successive Over-Relaxation} (SOR) \emph{method}~\cite{saad2003iterative}. Nonetheless, in order to make these iterative algorithms converge, one has to provide some constraints during the tensor update at each iteration. For example, the updated tensor is required to be positive-definite and/or diagonally dominant~\cite{li2017splitting, cui2019preconditioned, qi2018tensor, qi2017tensor}. When the tensor $\mathcal{A}$ is a special type of tensor, namely $\mathcal{M}$-tensors, the Eq.~\eqref{eq: tensor eq for introduction} becomes a $\mathcal{M}$-equation. Ding and Wei~\cite{ding2016solving} prove that a nonsingular $\mathcal{M}$-equation with a positive $\mathcal{B}$ always has a unique positive solution. Several iterative algorithms are proposed to solve multilinear nonsingular $\mathcal{M}$-equations by generalizing the classical iterative methods and the Newton method for linear systems. Furthermore, they also apply the $\mathcal{M}$-equations to solve nonlinear differential equations. In~\cite{xie2018tensor}, the authors solve these multilinear system of equations, especially focusing on symmetric $\mathcal{M}$-equations, by proposeing the rank-1 approximation of the tensor $\mathcal{A}$ and apply iterative tensor method to solve symmetric $\mathcal{M}$-equations. Their numerical examples demonstrate that the tensor methods could be more efficient than the Newton method for some $\mathcal{M}$-equations.

Sometimes, it is difficult to set a proper value of the underlying parameter (such as step size) in the solution-update equation to accelerate the convergence speed, while people often apply heuristics to determine such a parameter case by case. The other approach is to solve the unknown tensor $\mathcal{X}$ at the Eq.~\eqref{eq: tensor eq for introduction} through the tensor inversion. Brazell et al.~\cite{brazell2013solving} proposed the concept of the inverse of an even-order square tensor by adopting Einstein product, which provides a new direction to study tensors and tensor equations that model many phenomena in engineering and science~\cite{lai2009introduction}. In~\cite{bu2014inverse}, the authors give some basic properties for the left (right) inverse, rank and product of tensors. The existence of order 2 left (right) inverses of tensors is also characterized and several tensor properties, e.g., some equalities and inequalities on the tensor rank, independence between the rank of a uniform hypergraph and the ordering of its vertices, rank characteristics of the Laplacian tensor, are established through inverses of tensors. Since the key step in solving the Eq.~\eqref{eq: tensor eq for introduction} is to characterize the inverse of the tensor $\mathcal{A}$, Sun et al. in~\cite{sun2018generalized} define different types of inverse, namely, $i$-inverse ($i = 1, 2, 5$) and group inverse of tensors based on a general product of tensors. They explore properties of the generalized inverses of tensors on solving tensor equations and computing formulas of block tensors. The representations for the 1-inverse and group inverse of some block tensors are also established. They then use the 1-inverse of tensors to give the solutions of a multilinear system represented by tensors. The authors in~\cite{sun2018generalized} also proved that, for a tensor equation with invertible tensor $\mathcal{A}$, the solution is unique and can be expressed by the inverse of the tensor $\mathcal{A}$. 

However, the coefficient tensor $\mathcal{A}$ in the Eq.~\eqref{eq: tensor eq for introduction} is not always invertible, for example, when the tensor $\mathcal{A}$ is not square. Sun et al.~\cite{sun2016moore} extend the tensor inverse proposed by Brazell et al.~\cite{brazell2013solving} to the Moore–Penrose inverse via Einstein product, and a concrete representation for the Moore–Penrose inverse can be obtained by utilzing the singular value decomposition (SVD) of the tensor. An important application of the Moore–Penrose inverse is the tensor nearness problem associated with tensor equation with Einstein product, which can be expressed as follows~\cite{liang2019further}. Let $\mathcal{X}_0$ be a given tensor, find the tensor $\hat{\mathcal{X}} \in \Omega$ such that 
\begin{eqnarray}\label{eq: tensor nearness problem}
\left\Vert \hat{\mathcal{X}} - \mathcal{X}_0 \right\Vert &=& \min\limits_{\mathcal{X} \in \Omega} \left\Vert \mathcal{X} - \mathcal{X}_0 \right\Vert,
\end{eqnarray}
where $\left\Vert \cdot \right\Vert$ is the Frobenius norm, and $\Omega$ is the solution set of tensor equation
\begin{eqnarray}\label{eq: tensor nearness problem constraints}
\mathcal{A}\star_N\mathcal{X} &=& \mathcal{B}.
\end{eqnarray}
The tensor nearness problem is a generalization of the matrix nearness problem that are studied in many areas of applied matrix computations~\cite{dhillon2008matrix, guglielmi2015low}. The tensor $\mathcal{X}_0$ in Eq.~\eqref{eq: tensor nearness problem}, may be obtained by experimental measurement values and statistical distribution information, but it may not satisfy the desired form and the minimum error requirement, while the optimal estimation $\hat{\mathcal{X}}$ is the tensor that not only satisfies these restrictions but also best approximates $\mathcal{X}_0$. Under certain conditions, it will be proved that the solution to the tensor nearness problem~\eqref{eq: tensor nearness problem} is unique, and can be represented by means of the Moore–Penrose inverses of the known tensors~\cite{behera2017further}. Another situation to apply Moore-Penrose inverse is that the the given tensor equation in Eq.~\eqref{eq: tensor eq for introduction} has a non-square coefficient tensor $\mathcal{A}$. The associated least-squares problem for the solution in Eq.~\eqref{eq: tensor eq for introduction} can be obtained by solving the Moore-Penrose inverse of the coefficient tensor~\cite{brazell2013solving}. Several works to discuss the construction of Moore-Penrose inverse and how to apply this inverse to build necessary and sufficient conditions for the existence of the solution of tensor equations can be found at~\cite{ji2018drazin, jin2017generalized, behera2017further}.


In matrix theory, the Sherman–Morrison–Woodbury identity says that the inverse of a rank-$k$ correction of some matrix can be obtained by computing a rank-$k$ correction to the inverse of the original matrix. The Sherman–Morrison–Woodbury identity for matrix can be stated as following:
\begin{eqnarray}
(\mathbf{A} + \mathbf{U} \mathbf{B} \mathbf{V} )^{-1} = \mathbf{A}^{-1} -
\mathbf{A}^{-1} \mathbf{U}( \mathbf{B}^{-1} + \mathbf{V} \mathbf{A}^{-1} \mathbf{U} )^{-1} \mathbf{V} \mathbf{A}^{-1},
\end{eqnarray}
where $\mathbf{A}$ is a $n \times n$ matrix, $\mathbf{U}$ is a $n \times k$ matrix, $\mathbf{B}$ is a $k \times k$ matrix, and $\mathbf{V}$ is a $k \times n$ matrix. This identity is useful in numerical computations when $\mathbf{A}^{-1}$ has already been computed but the goal is to compute $(\mathbf{A} + \mathbf{U} \mathbf{B} \mathbf{V} )^{-1}$. With the inverse of $\mathbf{A}$ available, it is only necessary to find the inverse of $\mathbf{B}^{-1} + \mathbf{V} \mathbf{A}^{-1} \mathbf{U}$ in order to obtain the result using the right-hand side of the identity. If the matrix $\mathbf{B}$ has a much smaller dimension than $\mathbf{A}$, this is much easier than inverting $\mathbf{A} + \mathbf{U} \mathbf{B} \mathbf{V}$ directly. A common application is finding the inverse of a low-rank update $\mathbf{A} + \mathbf{U} \mathbf{B} \mathbf{V}$ of $\mathbf{A}$ when $\mathbf{U}$ only has a few columns and $\mathbf{V}$ also has only a few rows, or finding an approximation of the inverse of the matrix $\mathbf{A}+\mathbf{C}$ where the matrix $\mathbf{C}$ can be approximated by a low-rank matrix $\mathbf{U} \mathbf{B} \mathbf{V}$ via the singular value decomposition (SVD).

Analogously, we expect to have the Sherman–Morrison–Woodbury identity for tensors to facilitate the tensor inversion computation with those benefits in the matrix inversion computation when the correction of the original tensors is required. The Sherman–Morrison–Woodbury identity for tensors can be applied at various engineering and scientific areas, e.g., the tensor Kalman filter and recursive least squares methods~\cite{batselier2017tensor}. This identity can significantly speeds up the real time calculations of the tensor filter update because each new observation, which can be described with much lower dimension, can be treated as perturbation of the original covariance tensor. Similar to sensitivity analysis for linear systems~\cite{deif2012sensitivity}, if we wish to consider how the solution is affected by the perturbed of coefficients in the tensor $\mathcal{A}$ in Eq.~\eqref{eq: tensor eq for introduction}, we need to understand the relationship between the original tensor inverse and the perturbed tensor inverse. The Sherman–Morrison–Woodbury identity helps us to quantify the difference between the original solution and the perturbed solution of Eq.~\eqref{eq: tensor eq for introduction}. The contribution of this work can be summarized as follows.
\begin{enumerate}
\item We establish Sherman–Morrison–Woodbury identity for invertible tensors.
\item Because not every tensors are invertible, we generalize the Sherman–Morrison\\ 
-Woodbury identity for tensors with Moore-Penrose inverse. 
\item The sensitivity analysis is provided to the solution of a multilinear system 
when coefficient tensors are perturbed.  
\end{enumerate}

The paper is organized as follows. Preliminaries of tensors are given in Section~\ref{sec: Preliminaries of Tensors}.
In Section~\ref{sec:Identity for Invertible Tensors}, we will derive the Sherman–Morrison–Woodbury identity
for invertible tensors. In Section~\ref{sec: Identity for Tensors with Moore-Penrose Inverse}, the Sherman–Morrison–Woodbury identity is generalized for Moore-Penrose tensor inverse, and two illustrative examples about applying this identity are also presented. We apply Sherman–Morrison–Woodbury identity to analyze the sensitivity of perturbed multiplinear systems in Section~\ref{sec:Sensitivity Analysis for Multilinear Systems}. Finally, the conclusions are given in Section~\cref{sec:Conclusions}.

\section{Preliminaries of Tensors}\label{sec: Preliminaries of Tensors}

In this work, we denote scalars by lower-case letters (e.g., $d, e, f$), vectors by boldface lower-case letters (e.g., $\mathbf{d}, \mathbf{e}, \mathbf{f}$), matrices by boldface capital letters (e.g., $\mathbf{D},  \mathbf{E}, \mathbf{F}$), and tensors by calligraphic letters (e.g.,   $\mathcal{D},  \mathcal{E}, \mathcal{F}$), respectively.
Tensors are multiarray of values which are higher-dimensional generlization of vectors and matrices.
Given a positive integer $N$, let $[N] = {1, \cdots ,N}$. An order $N$ tensor $\mathcal{A}= (a_{i_1, \cdots, i_N})$, 
where $1 \leq i_j \leq I_j$ for $j \in [N]$, is a multidimensional
array with $I_1 \times I_2 \times \cdots \times I_{N}$ entries. Let $\mathbb{C}^{I_1 \times \cdots \times I_N}$ and $\mathbb{R}^{I_1 \times \cdots \times I_N}$ be the sets of the
order $N$ dimension $I_1 \times \cdots \times I_N$ tensors over the complex field $\mathbb{C}$ and the real
field $\mathbb{R}$, respectively.  For example, $\mathcal{A} \in \mathbb{C}^{I_1 \times \cdots \times I_N}$ is a multiway array with $N$-th order and $I_1, I_2, \cdots , I_N$ dimension in the first, second, · · · , $N$th direction,
respectively. Each entry of $\mathcal{A}$ is represented by $a_{i_1, \cdots, i_N}$. For $N = 4$, $\mathcal{A} \in \mathbb{C}^{I_1 \times I_2 \times I_3 \times I_4}$ is a fourth order tensor with entries as $a_{i_1, i_2, i_3, i_4}$.

For tensors $\mathcal{A} = (a_{i_1, \cdots, i_M, j_1, \cdots,j_N}) \in \mathbb{C}^{I_1 \times \cdots \times I_M\times
J_1 \times \cdots \times J_N}$ and \\
$\mathcal{B} = (a_{i_1, \cdots, i_M, j_1, \cdots,j_N}) \in \mathbb{C}^{I_1 \times \cdots \times I_M\times
J_1 \times \cdots \times J_N}$, the \emph{tensor addition} is defined as 
\begin{eqnarray}\label{eq: tensor addition definition}
(\mathcal{A} + \mathcal{B} )_{i_1, \cdots, i_M, j_1, \cdots, j_N} &=&
 a_{i_1, \cdots, i_M,  j_1, \cdots, j_N} + b_{i_1, \cdots, i_M,  j_1, \cdots, j_N}. 
\end{eqnarray}
If $M=N$ for the tensor $\mathcal{A} = (a_{i_1, \cdots, i_M, j_1, \cdots,j_N}) \in \mathbb{C}^{I_1 \times \cdots \times I_M\times
J_1 \times \cdots \times J_N}$, the tensor $\mathcal{A}$ is named as a \emph{square tensor}.

For tensors $\mathcal{A} = (a_{i_1, \cdots, i_M, j_1, \cdots,j_N}) \in \mathbb{C}^{I_1 \times \cdots \times I_M\times
J_1 \times \cdots \times J_N}$ and \\
$\mathcal{B} = (b_{j_1, \cdots, j_N, k_1, \cdots,k_L}) \in \mathbb{C}^{J_1 \times \cdots \times J_N\times K_1 \times \cdots \times K_L}$, the \emph{Einstein product} with order $N$ $\mathcal{A} \star_{N} \mathcal{B} \in  \mathbb{C}^{I_1 \times \cdots \times I_M\times
K_1 \times \cdots \times K_L}$ is defined as 
\begin{eqnarray}\label{eq: Einstein product definition}
(\mathcal{A} \star_{N} \mathcal{B} )_{i_1, \cdots, i_M,k_1, \cdots  ,k_L} &=& \sum\limits_{j_1, \cdots, j_N} a_{i_1, \cdots, i_M, j_1, \cdots,j_N}b_{j_1, \cdots, j_N, k_1, \cdots,k_L}. 
\end{eqnarray}
This tensor product reduces to the standard matrix multiplication when we have $L = M = N = 1$, which also contains the tensor–vector product and the tensor–matrix product as special cases.  
We need following definitions about tensors. We begin with zero tensor definition. 

\begin{definition}\label{def: zero tensor}
A tensor that all its entries are zero is called zero tensor, denoted as  $\mathcal{O}$. 
\end{definition}

The identity tensor is defined as following:
\begin{definition}\label{def: identity tensor}
An identity tensor $\mathcal{I} \in  \mathbb{C}^{I_1 \times \cdots \times I_N\times
J_1 \times \cdots \times J_N}$ is defined as 
\begin{eqnarray}\label{eq: identity tensor definition}
(\mathcal{I})_{i_1 \times \cdots \times i_N\times
j_1 \times \cdots \times j_N} = \prod_{k = 1}^{N} \delta_{i_k, j_k},
\end{eqnarray}
where $\delta_{i_k, j_k} = 1$ if $i_k  = j_k$, otherwise $\delta_{i_k, j_k} = 0$.
\end{definition}

In order to define a \emph{Hermitian} tensor, we need following conjugate transpose operation of a tensor.  

\begin{definition}\label{def: tensor conjugate transpose}
Let $\mathcal{A} = (a_{i_1, \cdots, i_M, j_1, \cdots,j_N}) \in \mathbb{C}^{I_1 \times \cdots \times I_M\times J_1 \times \cdots \times J_N}$ be a given tensor, then its conjugate transpose, denoted as
$\mathcal{A}^{H}$, is defined as 
\begin{eqnarray}\label{eq:tensor conjugate transpose definition}
(\mathcal{A}^H)_{ j_1, \cdots,j_N,i_1, \cdots, i_M}  =  
\overline{a_{i_1, \cdots, i_M,j_1, \cdots,j_N}},
\end{eqnarray}
where the over line indicates the complex conjugate of the complex number \\
$a_{i_1, \cdots, i_M,j_1, \cdots,j_N}$. If a tenser with the property $ \mathcal{A}^H = \mathcal{A}$, this tensor is named as \emph{Hermitian} tensor. 
\end{definition}

The meaning of the inverse of a tensor is provided as following: 
\begin{definition}\label{def: inverse of a tensor}
For  a square tensor $\mathcal{A} = (a_{i_1, \cdots, i_M, j_1, \cdots,j_M}) \in$ \\  $\mathbb{C}^{I_1 \times \cdots \times I_M\times I_1 \times \cdots \times I_M}$, if there exists $\mathcal{X} \in \mathbb{C}^{I_1 \times \cdots \times I_M\times I_1 \times \cdots \times I_M}$ such that 
\begin{eqnarray}\label{eq:tensor invertible definition}
\mathcal{A} \star_M \mathcal{X} = \mathcal{X} \star_M \mathcal{A} = \mathcal{I},
\end{eqnarray}
then such $\mathcal{X}$ is called as the inverse of the tensor $\mathcal{A}$, represented by $\mathcal{A}^{-1}$. 
\end{definition}

\begin{definition}\label{def: Moore-Penrose Inverse}
Given a tensor $\mathcal{A} \in \mathbb{C}^{I_1 \times \cdots \times I_M\times J_1 \times \cdots \times J_N}$,  then the tensor $\mathcal{X} \in  \mathbb{C}^{ J_1 \times \cdots \times J_N \times I_1 \times \cdots \times I_M}$, satisfying the following tensor equations:
\begin{eqnarray}\label{eq: Moore-Penrose Inverse rules}
(1) \mathcal{A} \star_N \mathcal{X} \star_M \mathcal{A} = \mathcal{A}, ~~~~~~
(2) \mathcal{X} \star_M \mathcal{A} \star_N \mathcal{X} = \mathcal{X}, ~~~~ \nonumber \\
(3) (\mathcal{A} \star_N \mathcal{X})^{H} = \mathcal{A} \star_M \mathcal{X}, ~~
(4) (\mathcal{X} \star_M \mathcal{A})^{H} = \mathcal{X} \star_M \mathcal{A},
\end{eqnarray}
is called the Moore-Penrose inverse of the tensor $\mathcal{A}$, denoted as $\mathcal{A}^{\dagger}$. 
\end{definition}

The trace of a tensor is defined as the summation of all the diagonal entries as 
\begin{eqnarray}\label{eq: tensor trace def}
\mathrm{Tr}(\mathcal{A}) = \sum\limits_{1 \leq i_j \leq I_j, j \in [N]} \mathcal{A}_{i_1, \cdots, i_M,i_1, \cdots, i_M}.
\end{eqnarray}
Then, we can define the inner product of two tensors $\mathcal{A}, \mathcal{B} \in \mathbb{C}^{I_1 \times \cdots \times I_M\times J_1 \times \cdots \times J_N}$ as 
\begin{eqnarray}\label{eq: tensor inner product def}
\langle \mathcal{A}, \mathcal{B} \rangle = \mathrm{Tr}(\mathcal{A}^H \star_M \mathcal{B}).
\end{eqnarray}
From the definition of tensor inner product, the Frobenius norm of a tensor $\mathcal{A}$ can be defined as 
\begin{eqnarray}\label{eq:Frobenius norm}
\left\Vert \mathcal{A} \right\Vert = \sqrt{\langle \mathcal{A}, \mathcal{A} \rangle}.
\end{eqnarray}

An unfolded tensor is a matrix obtained by reorganizing the entries of a tensor into a two-dimensional array. For the tensor space $\mathbb{C}^{I_1 \times \cdots \times I_M\times J_1 \times \cdots \times J_N}$ and the matrix space $\mathbb{C}^{(I_1 \cdots  I_M)\times (J_1 \cdots J_N)}$, we define a map $\varphi$ as follows:
\begin{eqnarray}\label{eq:unfold mapping}
\varphi: \mathbb{C}^{I_1 \times \cdots \times I_M\times J_1 \times \cdots \times J_N} \rightarrow 
\mathbb{C}^{(I_1 \cdots  I_M)\times (J_1 \cdots J_N)} \nonumber \\
\mathcal{A} = (a_{i_1, \cdots, i_M, j_1, \cdots,j_N}) \rightarrow (\mathbf{A}_{\phi(\mathbf{i}, \mathbb{I}), \phi(\mathbf{j}, \mathbb{J})}),
\end{eqnarray}
where $\phi$ is an index mapping function from tensor indices to matrix indices with arguments of row subscripts $\mathbf{i} = \{i_1, \cdots, i_M\}$ and row dimensions of $\mathcal{A}$, denoted as $\mathbb{I} = \{I_1, \cdots, I_M \}$. The relation $\phi(\mathbf{i}, \mathbb{I})$ can be expressed as 
\begin{eqnarray}\label{row indices mapping}
\phi(\mathbf{i}, \mathbb{I}) = i_1 + \sum\limits_{m=2}^{M}(i_m - 1) \prod\limits_{u = 1}^{m-1} I_u.
\end{eqnarray}
Similarly, $\phi(\mathbf{j}, \mathbb{J})$ is an index mapping relation for column dimensions of $\mathcal{A}$ which can be expressed as 
\begin{eqnarray}\label{column indices mapping}
\phi(\mathbf{j}, \mathbb{J}) = j_1 + \sum\limits_{n=2}^{N}(j_n - 1) \prod\limits_{v = 1}^{n-1} J_v,
\end{eqnarray}
where $\mathbf{j} = \{j_1, \cdots, j_N\}$ and column dimensions of $\mathcal{A}$, denoted as $\mathbb{J} = \{J_1, \cdots, J_N \}$. We will use this unfolding mapping $\varphi$ to build the condition of the existence of an inverse of a tensor. 

Following definition, which is based on the tensor unfolding map introduced by the Eq.~\eqref{eq:unfold mapping}, is required to determine when a given square tensor is invertible.
\begin{definition}\label{def: unfolding rank of tensor}
For a tensor $\mathcal{A} \in \mathbb{C}^{I_1 \times \cdots \times I_M\times J_1 \times \cdots \times J_N}$, and the map $\varphi$ defined by the Eq.~\eqref{eq:unfold mapping}, the unfolding rank of a tensor $\mathcal{A}$ is defined as the rank of the mapped matrix $\varphi(\mathcal{A})$. If we have $\varphi(\mathcal{A}) = I_1 \cdots I_M$ (the multiplication of all integers $I_1,  \cdots, I_M$ together), we say that $\mathcal{A}$ is full row rank. On ther other hand, if we have $\varphi(\mathcal{A}) = J_1 \cdots J_N$ (the multiplication of all integers $J_1,  \cdots, J_N$ together), we say that $\mathcal{A}$ is full column rank.
\end{definition}
Based on such unfolding rank definition, we are able to utilize this to give the sufficient and the necessary conditions for the existence of a given tensor provided by the following Lemma~\ref{lemma:condition for the existence of a tensor inverse}. 

\begin{lemma}\label{lemma:condition for the existence of a tensor inverse}
A given tensor $\mathcal{A} \in  \mathbb{C}^{I_1 \times \cdots \times I_M\times I_1 \times \cdots \times I_M}$ is invertible if and only if the matrix $\varphi(\mathcal{A})$ is a full rank matrix with rank value $I_1  \cdots I_M$. 
\end{lemma}
\textbf{Proof:}
See~\cite{liang2019further}. $\hfill \Box$

\section{Identity for Invertible Tensors}
\label{sec:Identity for Invertible Tensors}

The purpose of this section is to prove Sherman–Morrison–Woodbury identity for invertible tensors. 

\begin{theorem}[Sherman–Morrison–Woodbury identity for invertible tensors.]\label{thm:InvertibleWoodbury}
Given invertible tensors $\mathcal{A} \in  \mathbb{C}^{I_1 \times \cdots \times I_M\times I_1 \times \cdots \times I_M}$ and 
$\mathcal{B} \in  \mathbb{C}^{I_1 \times \cdots \times I_K\times I_1 \times \cdots \times I_K}$, and tensors $\mathcal{U} \in  \mathbb{C}^{I_1 \times \cdots \times I_M\times I_1 \times \cdots \times I_K}$ and $\mathcal{V} \in  \mathbb{C}^{I_1 \times \cdots \times I_K \times I_1 \times \cdots \times I_M}$, if the tensor $(\mathcal{B}^{-1} + \mathcal{V} \star_M \mathcal{A}^{-1} \star_M \mathcal{U})$ is invertible, we have following identiy:
\begin{eqnarray}\label{eq: Sherman Morrison Woodbury identity for invertible tensors}
(\mathcal{A} + \mathcal{U} \star_K \mathcal{B} \star_K \mathcal{V} )^{-1} = \mathcal{A}^{-1} - ~~~~~~~~~~~~~~~~~~~~~~~~~~~~~~~~~~~~~~~~~~~~~~~~~~~~~~~~ \nonumber \\
\mathcal{A}^{-1} \star_M \mathcal{U} \star_K ( \mathcal{B}^{-1} + \mathcal{V} \star_M \mathcal{A}^{-1} \star_M \mathcal{U} )^{-1} \star_K \mathcal{V} \star_M \mathcal{A}^{-1}.
\end{eqnarray}
\end{theorem}
\textbf{Proof: }
The identiy can be proven by checking that $(\mathcal{A} + \mathcal{U} \star_K \mathcal{B} \star_K \mathcal{V})$ multiplies its alleged inverse on the right side of the Sherman–Morrison–Woodbury identity gives the identity matrix (To save space, we omit Einstein product symbol,  $\star$, between two tensors.):
\begin{eqnarray}
(\mathcal{A} + \mathcal{U} \mathcal{B} \mathcal{V}) [\mathcal{A}^{-1} - \mathcal{A}^{-1} \mathcal{U} (\mathcal{B}^{-1}   +   \mathcal{V} \mathcal{A}^{-1} \mathcal{U} )^{-1} \mathcal{V} \mathcal{A}^{-1}]  ~~~~~~~~~~~~~~~~~~~~~~~~~~~~~~~~~~~~~~~~~~~~~~
\nonumber \\
 = \mathcal{I} + \mathcal{U}\mathcal{B}\mathcal{V}\mathcal{A}^{-1} - \mathcal{U}(\mathcal{B}^{-1} + \mathcal{V} \mathcal{A}^{-1}\mathcal{U})^{-1}\mathcal{V}\mathcal{A}^{-1} - 
\mathcal{U}\mathcal{B}\mathcal{V}\mathcal{A}^{-1}\mathcal{U}(\mathcal{B}^{-1} + \mathcal{V} \mathcal{A}^{-1}\mathcal{U})^{-1}\mathcal{V}\mathcal{A}^{-1}   ~~~~~ \nonumber \\
= (\mathcal{I} + \mathcal{U}\mathcal{B}\mathcal{V}\mathcal{A}^{-1}) - 
\left[  \mathcal{U}(\mathcal{B}^{-1} + \mathcal{V} \mathcal{A}^{-1}\mathcal{U})^{-1}\mathcal{V}\mathcal{A}^{-1} + \mathcal{U}\mathcal{B}\mathcal{V}\mathcal{A}^{-1}\mathcal{U}(\mathcal{B}^{-1} + \mathcal{V} \mathcal{A}^{-1}\mathcal{U})^{-1}\mathcal{V}\mathcal{A}^{-1}  \right]  \nonumber \\
= \mathcal{I} + \mathcal{U}\mathcal{B}\mathcal{V}\mathcal{A}^{-1} - (\mathcal{U} + \mathcal{U}\mathcal{B}\mathcal{V}\mathcal{A}^{-1}\mathcal{U})(\mathcal{B}^{-1} + \mathcal{V} \mathcal{A}^{-1} \mathcal{U})^{-1} \mathcal{V} \mathcal{A}^{-1}
 ~~~~~~~~~~~~~~~~~~~~~~~~~~~~~~~~~~~\nonumber \\
= \mathcal{I} + \mathcal{U}\mathcal{B}\mathcal{V}\mathcal{A}^{-1} - (\mathcal{U}\mathcal{B}(\mathcal{B}^{-1} + \mathcal{V}\mathcal{A}^{-1}\mathcal{U}) (\mathcal{B}^{-1} + \mathcal{V} \mathcal{A}^{-1} \mathcal{U})^{-1} \mathcal{V} \mathcal{A}^{-1}~~~~~~~~~~~~~~~~~~~~~~~~~~~~~~~ \nonumber \\
= \mathcal{I} + \mathcal{U}\mathcal{B}\mathcal{V}\mathcal{A}^{-1} - \mathcal{U}\mathcal{B}\mathcal{V}\mathcal{A}^{-1} = \mathcal{I} ~~~~~~~~~~~~~~~~~~~~~~~~~~~~~~~~~~~~~~~~~~~~~
\end{eqnarray}

Similar steps can be applied to prove this identity by multiplying the alleged inverse from the left side of $(\mathcal{A} + \mathcal{U} \mathcal{B} \mathcal{V})$:
\begin{eqnarray}
[\mathcal{A}^{-1} - \mathcal{A}^{-1} \mathcal{U} (\mathcal{B}^{-1}   +   \mathcal{V} \mathcal{A}^{-1} \mathcal{U} )^{-1} \mathcal{V} \mathcal{A}^{-1}] (\mathcal{A} + \mathcal{U} \mathcal{B} \mathcal{V}) ~~~~~~~~~~~~~~~~~~~~~~~~~~~~~~~~~~~~~~~~
\nonumber \\
 = \mathcal{I} + \mathcal{A}^{-1}\mathcal{U} (\mathcal{B}^{-1}   +   \mathcal{V} \mathcal{A}^{-1} \mathcal{U} )^{-1})^{-1}\mathcal{V} +  \mathcal{A}^{-1} \mathcal{U} \mathcal{B} \mathcal{V} - 
~~~~~~~~~~~~~~~~~~~~~~~~~~~~ \nonumber \\ 
\mathcal{A}^{-1} \mathcal{U} (\mathcal{B}^{-1}   +   \mathcal{V} \mathcal{A}^{-1} \mathcal{U} )^{-1})^{-1} \mathcal{V}\mathcal{A}^{-1} \mathcal{U} \mathcal{B} \mathcal{V} ~~~~~~~~~~~~~~~~~~~~~~~~~~~~~~~~~~~~~~~~ \nonumber \\
= (\mathcal{I} + \mathcal{A}^{-1}\mathcal{U}\mathcal{B}\mathcal{V}) - 
\mathcal{A}^{-1}\mathcal{U}(\mathcal{B}^{-1} + \mathcal{V} \mathcal{A}^{-1}\mathcal{U})^{-1}(\mathcal{V} + \mathcal{V}\mathcal{A}^{-1}\mathcal{U}\mathcal{B} \mathcal{V} ) ~~~~~~~~~~~~~~ \nonumber \\
= (\mathcal{I} + \mathcal{A}^{-1}\mathcal{U}\mathcal{B}\mathcal{V}) - 
\mathcal{A}^{-1}\mathcal{U}(\mathcal{B}^{-1} + \mathcal{V} \mathcal{A}^{-1}\mathcal{U})^{-1}(\mathcal{B}^{-1} + \mathcal{V}\mathcal{A}^{-1}\mathcal{U}) \mathcal{B}\mathcal{V} 
~~~~~~~~~~~ \nonumber \\
= (\mathcal{I} + \mathcal{A}^{-1}\mathcal{U}\mathcal{B}\mathcal{V})  -  
 \mathcal{A}^{-1}\mathcal{U}\mathcal{B}\mathcal{V} = \mathcal{I}  ~~~~~~~~~~~~~~~~~~~~~~~~~~~~~~~~~~~~~~~~~~~~~~~~~~~
\end{eqnarray}
Therefore, the identiy is established.   $\hfill \Box$

\section{Identity for Tensors with Moore-Penrose Inverse}
\label{sec: Identity for Tensors with Moore-Penrose Inverse}

In this section, we will extend our tensor inverse result from previous section to the Sherman–Morrison– 
 Woodbury identity for Moore-Penrose inverse in section~\ref{sec: Identity for Moore-Penrose Inverse Tensors}. Two illustrative examples for the Sherman–Morrison–Woodbury identity for Moore-Penrose inverse will be provided in section~\ref{sec: Example}.

\subsection{Identity for Moore-Penrose Inverse Tensors}
\label{sec: Identity for Moore-Penrose Inverse Tensors}

The goal of this section is to establish our main result: the Sherman–Morrison–Woodbury identity for Moore-Penrose inverse. We begin with the definitions about row space and column space of a given tensor. Let us define two symbols $\mathbb{I}_{M} \define \underbrace{1 \times \cdots \times 1 }_M$ and 
$\mathbb{I}_{N} \define \underbrace{1 \times \cdots \times 1}_N$. 

We define \emph{row-tensors} of a tensor $\mathcal{A} = (a_{i_1, \cdots, i_M, j_1, \cdots,j_N}) \in \mathbb{C}^{I_1 \times \cdots \times I_M\times
J_1 \times \cdots \times J_N}$ as subtensors $\mathbf{a}_{ i_1,\cdots,i_M}$ where $1 \leq i_k \leq I_k$ for $k \in [M]$. The entries in the row-tensor $\mathbf{a}_{ i_1,\cdots,i_M}$ are entries $a_{i_1,\cdots,i_M, j_1\cdots, j_N}$ where $1 \leq j_k \leq J_k$ for $k \in [N]$ but fix the indices of $i_1,\cdots,i_M$. Similarly, \emph{column-tensors} of a tensor $\mathcal{A}$ are subtensors $\mathbf{a}_{ j_1,\cdots,j_N}$ where $1 \leq j_k \leq J_k$ for $k \in [N]$. The entries in the column-tensor $\mathbf{a}_{ j_1,\cdots,j_N}$ are entries $a_{i_1,\cdots,i_M, j_1\cdots, j_N}$ where $1 \leq i_k \leq I_k$ for $k \in [M]$ but fix the indices of $j_1,\cdots,j_N$.

Let the tensor $\mathcal{A} \in \mathbb{C}^{I_1 \times \cdots \times I_M\times
J_1 \times \cdots \times J_N}$.  The right null space is defined as 
\begin{eqnarray}\label{eq: right null space def}
\mathfrak{N}_R(\mathcal{A}) &\define& \left\{ \mathbf{z} \in  \mathbb{C}^{
J_1 \times \cdots \times J_N \times \mathbb{I}_{M} }: \mathcal{A} \star_N \mathbf{z} = \mathcal{O} \right\} 
\end{eqnarray} 
Then the row space of $\mathcal{A}$ is defined as 
\begin{eqnarray}\label{eq: row space def}
\mathfrak{R}(\mathcal{A}) &\define& \bigg \{ \mathbf{y} \in  \mathbb{C}^{
J_1 \times \cdots \times J_N \times \mathbb{I}_{M}}: \mathbf{y} = \sum\limits_{i_1,\cdots,i_M}
\mathbf{a}_{ i_1,\cdots,i_M }x_{i_1,\cdots,i_M},  \nonumber \\
& & \mbox{where $x_{i_1,\cdots,i_M} \in \mathbb{C}$
and $\mathbf{a}_{ i_1,\cdots,i_M }  \in  \mathbb{C}^{
J_1 \times \cdots \times J_N \times \mathbb{I}_{M}}$.} \bigg \} 
\end{eqnarray} 
Now from the definition of right null space we have $\mathbf{a}^H_{ i_1,\cdots,i_M }\mathbf{z} = \mathcal{O}$, where $H$ is the Hermitian operator and $\mathbf{a}^H_{ i_1,\cdots,i_M }\in \mathbb{C}^{ \mathbb{I}_M \times J_1 \times \cdots \times J_N }$.  If we take any tensor $\mathbf{y} \in \mathfrak{R}(\mathcal{A}) $, then $\mathbf{y} = \sum\limits_{i_1,\cdots,i_M}
\mathbf{a}_{ i_1,\cdots,i_M }x_{i_1,\cdots,i_M}$, where $x_{i_1,\cdots,i_M} \in \mathbb{C}$.  Hence, 
\begin{eqnarray}\label{eq: row space orthgonal to right null}
\mathbf{y}^{H} \mathbf{z} &=&  (  \sum\limits_{i_1,\cdots,i_M}
\mathbf{a}_{ i_1,\cdots,i_M }x_{i_1,\cdots,i_M} )^{H} \mathbf{z} \nonumber \\
&=&  (  \sum\limits_{i_1,\cdots,i_M} x_{i_1,\cdots,i_M}
\mathbf{a}^{H}_{ i_1,\cdots,i_M } ) \mathbf{z} \nonumber \\
&=&   \sum\limits_{i_1,\cdots,i_M} x_{i_1,\cdots,i_M}
(\mathbf{a}^{H}_{ i_1,\cdots,i_M }  \mathbf{z}) = \mathcal{O}
\end{eqnarray} 
This shows that row space is orthogonal to the right null space. 

Following this right null space approach, we also can define the left null space as 
\begin{eqnarray}\label{eq: left null space def}
\mathfrak{N}_L(\mathcal{A}) &\define& \left\{ \mathbf{z} \in  \mathbb{C}^{
I_1 \times \cdots \times I_M \times \mathbb{I}_{N} }:  \mathbf{z}^{H}  \star_M \mathcal{A} = \mathcal{O} \right\} 
\end{eqnarray} 
Then the column space of $\mathcal{A}$ is defined as 
\begin{eqnarray}\label{eq: column space def}
\mathfrak{C}(\mathcal{A}) &\define& \bigg \{ \mathbf{y} \in  \mathbb{C}^{
I_1 \times \cdots \times I_M \times \mathbb{I}_{N} }: \mathbf{y} = \sum\limits_{j_1,\cdots,j_N}
\mathbf{a}_{ j_1,\cdots,j_N }x_{j_1,\cdots,j_N},  \nonumber \\
& & \mbox{where $x_{j_1,\cdots,j_N} \in \mathbb{C}$
and $\mathbf{a}_{ j_1,\cdots,j_N }  \in  \mathbb{C}^{
I_1 \times \cdots \times I_M \times \mathbb{I}_{N} }$.} \bigg \} 
\end{eqnarray} 
From the definition of left null space we have $\mathbf{z}^{H}\mathbf{a}_{ j_1,\cdots,j_N } = \mathcal{O}$, where $H$ is the Hermitian operator and $\mathbf{z}^{H} \in \mathbb{C}^{\mathbb{I}_{N}  \times  I_1 \times \cdots \times I_M }$. By taking any tensor $\mathbf{y} \in \mathfrak{C}(\mathcal{A}) $, then $\mathbf{y} = \sum\limits_{j_1,\cdots,j_N}
\mathbf{a}_{ j_1,\cdots, j_N }x_{j_1,\cdots,j_N}$, where $x_{j_1,\cdots,j_N} \in \mathbb{C}$.  We have, 
\begin{eqnarray}\label{eq: column space orthgonal to left null}
\mathbf{z}^{H}\mathbf{y}  &=&  \mathbf{z}^{H}(  \sum\limits_{j_1,\cdots,j_N}
\mathbf{a}_{ j_1,\cdots,j_N }x_{j_1,\cdots,j_N} )  \nonumber \\
&=&  (   \sum\limits_{j_1,\cdots,j_N} \mathbf{z}^{H}
\mathbf{a}_{ j_1,\cdots,j_N } x_{j_1,\cdots,j_N} ) \nonumber \\
&=&   \sum\limits_{ j_1,\cdots,j_N} x_{i_1,\cdots,i_M}
(\mathbf{z}^{T}  \mathbf{a}_{  j_1,\cdots,j_N } ) = \mathcal{O}
\end{eqnarray} 
This shows that column space is orthogonal to the left null space. 

Given following tensor relation:
\begin{eqnarray}\label{eq: S tensor sum as A B}
\mathcal{S}&=& \mathcal{A} +  \mathcal{U} \mathcal{B} \mathcal{V},  
\end{eqnarray}
the goal is to expresse the Moore-Penrose inverse of $\mathcal{S}$ in terms 
of tensors related to $\mathcal{A}, \mathcal{U}, \mathcal{B}, \mathcal{V}$. From
the definition of column space, we can decompose the tensor $\mathcal{U}$ into $\mathcal{X}_1 + \mathcal{Y}_1$, wherer the column-tensors of $\mathcal{X}_1$ are contained in the clumn space of $\mathcal{A}$, denoted as $\mathfrak{C}(\mathcal{A})$, and the column-tensors of $\mathcal{Y}_1$ are contained in the left null space of $\mathcal{A}$. Correspondingly, we also can decompose the tensor $\mathcal{V}^{H}$ into $\mathcal{X}_2 + \mathcal{Y}_2$, wherer the column-tensors of $\mathcal{X}_2$ are contained in the clumn space of $\mathcal{A}^{H}$, denoted as $\mathfrak{C}(\mathcal{A}^{H})$, and the column-tensors of $\mathcal{Y}_2$ are contained in the left null space of $\mathcal{A}^{H}$. Define tensors $\mathcal{E}_i$ as $\mathcal{E}_i \define \mathcal{Y}_i ( \mathcal{Y}^{H}_i \mathcal{Y}_i )^{\dagger}$ for $i =1, 2$. We are ready to present the following theorem about the identity for tensors with Moore-Penrose inverse.

\begin{theorem}[Sherman–Morrison–Woodbury identity for Moore-Penrose inverse]
\label{thm: Identity for Tensors with Moore-Penrose Inverse}
Given tensors $\mathcal{A} \in  \mathbb{C}^{I_1 \times \cdots \times I_M\times I_1 \times \cdots \times I_N}$, $\mathcal{B} \in  \mathbb{C}^{I_1 \times \cdots \times I_K\times I_1 \times \cdots \times I_K}$, \nonumber \\
$\mathcal{U} \in  \mathbb{C}^{I_1 \times \cdots \times I_M\times I_1 \times \cdots \times I_K}$ and $\mathcal{V} \in  \mathbb{C}^{I_1 \times \cdots \times I_K \times I_1 \times \cdots \times I_N}$, if following conditions are satisfied:
\begin{enumerate}\label{enu: MP inverse conditions}
  \item $\mathcal{U} = \mathcal{X}_1 + \mathcal{Y}_1$, where $\mathcal{X}_1 \in \mathfrak{C}(\mathcal{A})$ and $\mathcal{Y}_1$ is orthgonal to $\mathfrak{C}(\mathcal{A})$;
  \item $\mathcal{V}^{H} = \mathcal{X}_2 + \mathcal{Y}_2$, where $\mathcal{X}_2 \in \mathfrak{C}(\mathcal{A}^{H})$ and $\mathcal{Y}_2$ is orthgonal to $\mathfrak{C}(\mathcal{A}^{H})$;
 \item (1) $\mathcal{E}_2 \mathcal{B}^{\dagger} \mathcal{E}_1^{H}\mathcal{Y}_1\mathcal{B} = \mathcal{E}_2$, (2) $\mathcal{X}_1 \mathcal{E}_1^{H} \mathcal{Y}_1\mathcal{B} = \mathcal{X}_1\mathcal{B}$, (3) $\mathcal{Y}_1 \mathcal{E}^{H}_1 \mathcal{Y}_1 =  \mathcal{Y}_1$;
 \item (1) $\mathcal{B} \mathcal{Y}_2^{H} \mathcal{E}_2\mathcal{B}^{\dagger}\mathcal{E}_1^{H} = \mathcal{E}_1^{H}$, (2)  $\mathcal{B} \mathcal{Y}_2^{H} \mathcal{E}_2\mathcal{X}_2^{H} = \mathcal{B}\mathcal{X}_2^{H}$, (3) $\mathcal{E}_2 \mathcal{Y}^{H}_2 \mathcal{E}_2=  \mathcal{E}_2 $.
\end{enumerate}
Then the tensor  
%
\begin{eqnarray}\label{eq: S tensors}
\mathcal{S} &=& \mathcal{A} + \mathcal{U} \star_K \mathcal{B} \star_K \mathcal{V} \nonumber \\
&=& \mathcal{A} + (\mathcal{X}_1 + \mathcal{Y}_1 ) \star_K \mathcal{B} \star_K (\mathcal{X}_2 + \mathcal{Y}_2 )^H, 
\end{eqnarray}
has the following Moore-Penrose generalized inverse identiy:
\begin{eqnarray}\label{eq: Sherman Morrison Woodbury identity for Moore-Penrose inverse tensors}
\mathcal{S}^{\dagger} &=& \mathcal{A}^{\dagger} - \mathcal{E}_2 \star_K \mathcal{X}^{H}_2 \star_N \mathcal{A}^{\dagger} - \mathcal{A}^{\dagger} \star_M \mathcal{X}_1 \star_K \mathcal{E}_1^{H} \nonumber \\
& &+\mathcal{E}_2 \star_K (\mathcal{B}^{\dagger} + \mathcal{X}^{H}_2 \star_N \mathcal{A}^{\dagger} \star_M \mathcal{X}_1 ) \star_K \mathcal{E}^{H}_1,
\end{eqnarray}
where $\mathcal{E}_i \define \mathcal{Y}_i ( \mathcal{Y}^{H}_i \mathcal{Y}_i )^{\dagger}$ for $i =1, 2$. 

\end{theorem}
\textbf{Proof: }
From the definition~\ref{def: Moore-Penrose Inverse}, the identity is established by direct computation (To save space, we also omit $\star$ product symbol between two tensors in this proof) to verify following four rules: 
\begin{eqnarray}\label{eq:proof in MP identity 1}
(1) \mathcal{S} \mathcal{S}^{\dagger} \mathcal{S} = \mathcal{S}, ~~~~~~~~
(2) \mathcal{S}^{\dagger} \mathcal{S} \mathcal{S}^{\dagger} = \mathcal{S}^{\dagger}, ~~~ \nonumber \\
(3) (\mathcal{S} \mathcal{S}^{\dagger})^{H} = \mathcal{S} \mathcal{S}^{\dagger}, ~~
(4) (\mathcal{S}^{\dagger} \mathcal{S})^{H} = \mathcal{S}^{\dagger}\mathcal{S}.
\end{eqnarray}

\textbf{Verify : $(\mathcal{S} \mathcal{S}^{\dagger})^{H} = \mathcal{S} \mathcal{S}^{\dagger}$}

By expansion of $\mathcal{S}\mathcal{S}^{\dagger}$, we have 
\begin{eqnarray}\label{eq:proof in MP identity 2}
\mathcal{S}\mathcal{S}^{\dagger} &=&  \mathcal{A}\mathcal{A}^{\dagger} - \mathcal{A}\mathcal{E}_2 \mathcal{X}^{H}_2 \mathcal{A}^{\dagger} -   \mathcal{A}\mathcal{A}^{\dagger} \mathcal{X}_1 \mathcal{E}_1^{H} + \mathcal{A}\mathcal{E}_2( \mathcal{B}^{\dagger} + \mathcal{X}_2^{H} \mathcal{A}^{\dagger} \mathcal{X}_1     )\mathcal{E}^{H}_1 +  \nonumber \\
& & 
(\mathcal{X}_1 + \mathcal{Y}_1) \mathcal{B} (\mathcal{X}_2 + \mathcal{Y}_2)^{H} \mathcal{A}^{\dagger} -  (\mathcal{X}_1 + \mathcal{Y}_1) \mathcal{B} (\mathcal{X}_2 + \mathcal{Y}_2)^{H}\mathcal{E}_2 \mathcal{X}_2^{H} \mathcal{A}^{\dagger} - \nonumber \\
& &  (\mathcal{X}_1 + \mathcal{Y}_1) \mathcal{B} (\mathcal{X}_2 + \mathcal{Y}_2)^{H}\mathcal{A}^{\dagger}  \mathcal{X}_1 \mathcal{E}_1^{H}  +  (\mathcal{X}_1 + \mathcal{Y}_1) \mathcal{B} (\mathcal{X}_2 + \mathcal{Y}_2)^{H}\mathcal{E}_2 \mathcal{X}_2^{H} \mathcal{A}^{\dagger}  \mathcal{X}_1 \mathcal{E}_1^{H}  + \nonumber \\
& &  (\mathcal{X}_1 + \mathcal{Y}_1) \mathcal{B} (\mathcal{X}_2 + \mathcal{Y}_2)^{H}
\mathcal{E}_2 \mathcal{B}^{\dagger} \mathcal{E}_1^{H},
\end{eqnarray}
and, since column-tensors of $\mathcal{Y}_2$ are orthgonal to $\mathfrak{C}(\mathcal{A}^{H})$, we also have $\mathcal{A}\mathcal{Y}_2 = \mathcal{O}$, $\mathcal{Y}^{H}_2\mathcal{A}^{\dagger} = \mathcal{O}$, and $\mathcal{X}_2^{H}\mathcal{Y}_2 = \mathcal{O}$. From these relations, the Eq.~\eqref{eq:proof in MP identity 2} can be simplfies as 
\begin{eqnarray}\label{eq:proof in MP identity 3}
\mathcal{S}\mathcal{S}^{\dagger} &=&  \mathcal{A}\mathcal{A}^{\dagger} -   \mathcal{A}\mathcal{A}^{\dagger} \mathcal{X}_1 \mathcal{E}_1^{H} +  \nonumber \\
& & 
(\mathcal{X}_1 + \mathcal{Y}_1) \mathcal{B} \mathcal{X}_2 ^{H} \mathcal{A}^{\dagger} -  (\mathcal{X}_1 + \mathcal{Y}_1) \mathcal{B} \mathcal{Y}_2^{H}\mathcal{E}_2 \mathcal{X}_2^{H} \mathcal{A}^{\dagger} - \nonumber \\
& &  (\mathcal{X}_1 + \mathcal{Y}_1) \mathcal{B} \mathcal{X}_2^{H}\mathcal{A}^{\dagger}  \mathcal{X}_1 \mathcal{E}_1^{H}  +  (\mathcal{X}_1 + \mathcal{Y}_1) \mathcal{B} \mathcal{Y}_2^{H}\mathcal{E}_2 \mathcal{X}_2^{H} \mathcal{A}^{\dagger}  \mathcal{X}_1 \mathcal{E}_1^{H}  + \nonumber \\
& &  (\mathcal{X}_1 + \mathcal{Y}_1) \mathcal{B} \mathcal{Y}_2^{H}
\mathcal{E}_2 \mathcal{B}^{\dagger} \mathcal{E}_1^{H},
\end{eqnarray}
and from the fourth condition at this Theorem~\ref{thm: Identity for Tensors with Moore-Penrose Inverse}, i.e., $\mathcal{B} \mathcal{Y}_2^{H} \mathcal{E}_2\mathcal{B}^{\dagger}\mathcal{E}_1^{H} = \mathcal{E}_1^{H}$, and  $\mathcal{B} \mathcal{Y}_2^{H} \mathcal{E}_2\mathcal{X}_2^{H} = \mathcal{B}\mathcal{X}_2^{H}$ and $\mathcal{A}\mathcal{A}^{\dagger} \mathcal{X}_1 = \mathcal{X}_1$, we can further simplifies the Eq.~\eqref{eq:proof in MP identity 3} as :
\begin{eqnarray}\label{eq:proof in MP identity 4}
\mathcal{S}\mathcal{S}^{\dagger} &=&  \mathcal{A}\mathcal{A}^{\dagger} +  \mathcal{Y}_1 \mathcal{E}_1^{H}.
\end{eqnarray}
Then, 
\begin{eqnarray}\label{eq:proof in MP identity 5}
(\mathcal{S}\mathcal{S}^{\dagger})^H &=&  (\mathcal{A}\mathcal{A}^{\dagger} +  \mathcal{Y}_1 \mathcal{E}_1^{H})^H \nonumber \\
&=&
 (\mathcal{A}\mathcal{A}^{\dagger})^{H} +  (\mathcal{Y}_1 \mathcal{E}_1^{H})^H \nonumber \\
&=&
 \mathcal{A}\mathcal{A}^{\dagger} +  \mathcal{Y}_1 \mathcal{E}_1^{H} = \mathcal{S}\mathcal{S}^{\dagger}
\end{eqnarray}
the third requirement of the definition~\ref{def: Moore-Penrose Inverse} is valid from the Eq.~\eqref{eq:proof in MP identity 5}.

\textbf{Verify : $ (\mathcal{S}^{\dagger} \mathcal{S})^{H} = \mathcal{S}^{\dagger}\mathcal{S}$}

By expansion $\mathcal{S}^{\dagger} \mathcal{S}$ as the Eq.~\eqref{eq:proof in MP identity 2}, we can simplifies such expansion with following relations $\mathcal{A}^{\dagger} \mathcal{Y}_1 = \mathcal{O}$, $\mathcal{Y}^{H}_1\mathcal{A}  = \mathcal{O}$ and $\mathcal{X}^{H}_1 \mathcal{Y}_1 = \mathcal{O}$ due to column-tensors of $\mathcal{Y}_1$ are orthgonal to $\mathfrak{C}(\mathcal{A})$, $\mathcal{X}_2 \mathcal{A}^{\dagger} \mathcal{A} = \mathcal{X}_2$ (from the definition of Moore-Penrose inverse of $\mathcal{A}$), and the third condition at this Theorem~\ref{thm: Identity for Tensors with Moore-Penrose Inverse}, i.e., $\mathcal{E}_2 \mathcal{B}^{\dagger} \mathcal{E}_1^{H}\mathcal{Y}_1\mathcal{B} = \mathcal{E}_2$, and $\mathcal{X}_1 \mathcal{E}_1^{H} \mathcal{Y}_1\mathcal{B} = \mathcal{X}_1\mathcal{B}$, then we have 
\begin{eqnarray}\label{eq:proof in MP identity 6}
\mathcal{S}^{\dagger}\mathcal{S} &=&  \mathcal{A}^{\dagger} \mathcal{A} +  \mathcal{E}_2 \mathcal{Y}_2^{H}.
\end{eqnarray}
Hence, the fourth requirement of the definition~\ref{def: Moore-Penrose Inverse} is valid from the Eq.~\eqref{eq:proof in MP identity 6}.

\textbf{Verify : $\mathcal{S} \mathcal{S}^{\dagger} \mathcal{S} = \mathcal{S}$}

Since, we have 
\begin{eqnarray}\label{eq:proof in MP identity 7}
\mathcal{S} \mathcal{S}^{\dagger} \mathcal{S}&=& ( \mathcal{A}\mathcal{A}^{\dagger} +  \mathcal{Y}_1 \mathcal{E}_1^{H} ) (\mathcal{A} + (\mathcal{X}_1 + \mathcal{Y}_1 ) \mathcal{B}(\mathcal{X}_2 + \mathcal{Y}_2 )^H) \nonumber \\
&=& \mathcal{A} \mathcal{A}^{\dagger} \mathcal{A} + \mathcal{Y}_1\mathcal{E}^{H}_1\mathcal{A} + \mathcal{A} \mathcal{A}^{\dagger}  (\mathcal{X}_1 + \mathcal{Y}_1 ) \mathcal{B}(\mathcal{X}_2 + \mathcal{Y}_2 )^H + \nonumber \\
&  &\mathcal{Y}_1\mathcal{E}^{H}_1 (\mathcal{X}_1 + \mathcal{Y}_1 ) \mathcal{B}(\mathcal{X}_2 + \mathcal{Y}_2 )^H \nonumber \\
&=& \mathcal{A} +  \mathcal{Y}_1 \mathcal{E}^{H}_1\mathcal{A}+\mathcal{A} \mathcal{A}^{\dagger} \mathcal{X}_1\mathcal{B}(\mathcal{X}_2 + \mathcal{Y}_2 )^H + 
\mathcal{A} \mathcal{A}^{\dagger} \mathcal{Y}_1\mathcal{B}(\mathcal{X}_2 + \mathcal{Y}_2 )^H +\nonumber \\
& &  \mathcal{Y}_1 \mathcal{E}^{H}_1 \mathcal{X}_1\mathcal{B}(\mathcal{X}_2 + \mathcal{Y}_2 )^H + \mathcal{Y}_1 \mathcal{E}^{H}_1 \mathcal{Y}_1\mathcal{B}(\mathcal{X}_2 + \mathcal{Y}_2 )^H  \nonumber \\
&\stackrel{1}{=}&\mathcal{A} + \mathcal{X}_1\mathcal{B}(\mathcal{X}_2 + \mathcal{Y}_2 )^{H} +  \mathcal{Y}_1\mathcal{B}(\mathcal{X}_2 + \mathcal{Y}_2 )^{H} \nonumber \\
&=& \mathcal{A} + (\mathcal{X}_1 + \mathcal{Y}_1 )\mathcal{B}(\mathcal{X}_2 + \mathcal{Y}_2 )^{H}
=  \mathcal{S}
\end{eqnarray}
where we apply $\mathcal{A}\mathcal{A}^{\dagger}\mathcal{X}_1 = \mathcal{X}_1$,  $\mathcal{A}^{\dagger} \mathcal{Y}_1 = \mathcal{O}$, $\mathcal{Y}^{H}_1\mathcal{A}  = \mathcal{O}$, $\mathcal{Y}^{H}_1 \mathcal{X}_1 = \mathcal{O}$ and (3) of the third condition at this Theorem~\ref{thm: Identity for Tensors with Moore-Penrose Inverse}, i.e., $\mathcal{Y}_1 \mathcal{E}^{H}_1 \mathcal{Y}_1 =  \mathcal{Y}_1$ at the equality $\stackrel{1}{=}$ in simplification. 

\textbf{Verify : $\mathcal{S}^{\dagger} \mathcal{S} \mathcal{S}^{\dagger} = \mathcal{S}^{\dagger}$}

Since, we have 
\begin{eqnarray}\label{eq:proof in MP identity 8}
\mathcal{S}^{\dagger}  \mathcal{S}\mathcal{S}^{\dagger} &=& ( \mathcal{A}^{\dagger}\mathcal{A} +  \mathcal{E}_2 \mathcal{Y}_2^{H} ) ( \mathcal{A}^{\dagger} - \mathcal{E}_2 \mathcal{X}^{H}_2  \mathcal{A}^{\dagger} - \mathcal{A}^{\dagger}  \mathcal{X}_1 \mathcal{E}_1^{H} +\mathcal{E}_2 (\mathcal{B}^{\dagger} + \mathcal{X}^{H}_2 \mathcal{A}^{\dagger} \mathcal{X}_1 ) \mathcal{E}^{H}_1 ) \nonumber \\
&=&  \mathcal{A}^{\dagger}\mathcal{A} \mathcal{A}^{\dagger} + \mathcal{E}_2\mathcal{Y}_2^{H}\mathcal{A}^{\dagger} - \mathcal{A}^{\dagger}\mathcal{A}\mathcal{E}_2\mathcal{X}^{H}\mathcal{A}^{H} - \mathcal{E}_2 \mathcal{Y}_2^{H}  \mathcal{E}_2 \mathcal{X}^{H}\mathcal{A}^{H} - \mathcal{A}^{\dagger}\mathcal{A} \mathcal{A}^{\dagger}\mathcal{X}_1\mathcal{E}_1^{H} -
\nonumber \\
& &  \mathcal{E}\mathcal{Y}^{H}_2\mathcal{A}^{\dagger} \mathcal{X}_1\mathcal{E}_1^{H} + 
  \mathcal{A}^{\dagger}\mathcal{A}\mathcal{E}_2  (\mathcal{B}^{\dagger} + \mathcal{X}^{H}_2 \mathcal{A}^{\dagger} \mathcal{X}_1 ) \mathcal{E}^{H}_1 + \mathcal{E}_2\mathcal{Y}^{H}_2\mathcal{E}_2  (\mathcal{B}^{\dagger} + \mathcal{X}^{H}_2 \mathcal{A}^{\dagger} \mathcal{X}_1 ) \mathcal{E}^{H}_1 \nonumber \\
&\stackrel{2}{=}& \mathcal{A}^{\dagger} - \mathcal{E}_2 \mathcal{X}^{H}_2  \mathcal{A}^{\dagger} - \mathcal{A}^{\dagger}  \mathcal{X}_1 \mathcal{E}_1^{H} +\mathcal{E}_2 (\mathcal{B}^{\dagger} + \mathcal{X}^{H}_2 \mathcal{A}^{\dagger} \mathcal{X}_1 ) \mathcal{E}^{H}_1 = \mathcal{S}^{\dagger}
\end{eqnarray}
where we apply $\mathcal{A}\mathcal{A}^{\dagger}\mathcal{X}_1 = \mathcal{X}_1$,  $\mathcal{A}\mathcal{Y}_2 = \mathcal{O}$, $\mathcal{Y}^{H}_2\mathcal{A}^{\dagger} = \mathcal{O}$  and (3) of the fourth condition at this Theorem~\ref{thm: Identity for Tensors with Moore-Penrose Inverse}, i.e., $\mathcal{E}_2 \mathcal{Y}^{H}_2 \mathcal{E}_2 =  \mathcal{E}_2$ at the equality $\stackrel{2}{=}$ in simplification. 

Finally, since all requirements for Moore-Penrose generlized inverse definition have been fulfilled, this main theorem is proved.    $\hfill \Box$

There are several corollaries can be extended based on Theorem~\ref{thm: Identity for Tensors with Moore-Penrose Inverse}. If tensors $\mathcal{B}$ and $(\mathcal{Y}^{H}_i \mathcal{Y}_i )$ for $i =1, 2$ are invertible, those requirements in Theorem~\ref{thm: Identity for Tensors with Moore-Penrose Inverse} can be reduced. Then, we have following corollary based on the Theorem~\ref{thm: Identity for Tensors with Moore-Penrose Inverse}.

\begin{corollary}\label{thm: Identity for Tensors with Moore-Penrose Inverse YY invertible}
Given tensors $\mathcal{A} \in  \mathbb{C}^{I_1 \times \cdots \times I_M\times I_1 \times \cdots \times I_N}$, and the invertible tensor $\mathcal{B} \in  \mathbb{C}^{I_1 \times \cdots \times I_K\times I_1 \times \cdots \times I_K}$, $\mathcal{U} \in  \mathbb{C}^{I_1 \times \cdots \times I_M\times I_1 \times \cdots \times I_K}$ and $\mathcal{V} \in  \mathbb{C}^{I_1 \times \cdots \times I_K \times I_1 \times \cdots \times I_N}$, if following conditions are valid :
\begin{enumerate}\label{enu: MP inverse conditions  YY invertible}
  \item $\mathcal{U} = \mathcal{X}_1 + \mathcal{Y}_1$, where $\mathcal{X}_1 \in \mathfrak{C}(\mathcal{A})$ and $\mathcal{Y}_1$ is orthgonal to $\mathfrak{C}(\mathcal{A})$;
  \item $\mathcal{V}^{H} = \mathcal{X}_2 + \mathcal{Y}_2$, where $\mathcal{X}_2 \in \mathfrak{C}(\mathcal{A}^{H})$ and $\mathcal{Y}_2$ is orthgonal to $\mathfrak{C}(\mathcal{A}^{H})$;
 \item tensors $(\mathcal{Y}^{H}_i \mathcal{Y}_i )$ for $i=1, 2$ are invertible.
\end{enumerate}
Then the tensor  
%
\begin{eqnarray}\label{eq: S tensors YY invertible}
\mathcal{S} &=& \mathcal{A} + \mathcal{U} \star_K \mathcal{B} \star_K \mathcal{V} \nonumber \\
&=& \mathcal{A} + (\mathcal{X}_1 + \mathcal{Y}_1 ) \star_K \mathcal{B} \star_K (\mathcal{X}_2 + \mathcal{Y}_2 )^H, 
\end{eqnarray}
has the following Moore-Penrose generalized inverse identiy:
\begin{eqnarray}\label{eq: Sherman Morrison Woodbury identity for Moore-Penrose inverse tensors YY invertible}
\mathcal{S}^{\dagger} &=& \mathcal{A}^{\dagger} - \mathcal{E}_2 \star_K \mathcal{X}^{H}_2 \star_N \mathcal{A}^{\dagger} - \mathcal{A}^{\dagger} \star_M \mathcal{X}_1 \star_K \mathcal{E}_1^{H} \nonumber \\
& &+\mathcal{E}_2 \star_K (\mathcal{B}^{\dagger} + \mathcal{X}^{H}_2 \star_N \mathcal{A}^{\dagger} \star_M \mathcal{X}_1 ) \star_K \mathcal{E}^{H}_1,
\end{eqnarray}
where $\mathcal{E}_i \define \mathcal{Y}_i ( \mathcal{Y}^{H}_i \mathcal{Y}_i )^{-1}$ for $i =1, 2$. 
\end{corollary}
\textbf{Proof:} 

Because $ (\mathcal{Y}^{H}_i \mathcal{Y}_i )$ for $i=1, 2$ are invertible, then 
\begin{eqnarray}
\mathcal{E}_2 \mathcal{B}^{\dagger} \mathcal{E}_1^{H}\mathcal{Y}_1\mathcal{B} = 
\mathcal{E}_2 \mathcal{B}^{-1} \left[ \mathcal{Y}_1 (\mathcal{Y}^{H}_1 \mathcal{Y}_1)^{-1} \right]^{H}\mathcal{Y}_1\mathcal{B} = \mathcal{E}_2.
\end{eqnarray}
By similar arguments, all following conditions are valid:
\begin{itemize}
\item $\mathcal{X}_1 \mathcal{E}_1^{H} \mathcal{Y}_1\mathcal{B} = \mathcal{X}_1\mathcal{B}$, 
\item $\mathcal{Y}_1 \mathcal{E}^{H}_1 \mathcal{Y}_1 =  \mathcal{Y}_1$,
\item $\mathcal{B} \mathcal{Y}_2^{H} \mathcal{E}_2\mathcal{B}^{-1}\mathcal{E}_1^{H} = \mathcal{E}_1^{H}$,
\item $\mathcal{B} \mathcal{Y}_2^{H} \mathcal{E}_2\mathcal{X}_2^{H} = \mathcal{B}\mathcal{X}_2^{H}$, 
\item $\mathcal{E}_2 \mathcal{Y}^{H}_2 \mathcal{E}_2=  \mathcal{E}_2 $.
\end{itemize}
Therefore, this corollary is proved from Theorem~\ref{thm: Identity for Tensors with Moore-Penrose Inverse} because all conditions required at Theorem~\ref{thm: Identity for Tensors with Moore-Penrose Inverse} are satisfied. $\hfill \Box$

When the column space projections of the tensors $\mathcal{U}$ and $\mathcal{V}$ are zero, i.e., $\mathcal{X}_1 = \mathcal{O}$ and $\mathcal{X}_2 = \mathcal{O}$. Theorem~\ref{thm: Identity for Tensors with Moore-Penrose Inverse} can be simplfied as following corollary. 

\begin{corollary}\label{cor: Identity for Tensors with Moore-Penrose Inverse X zero}
Given tensors $\mathcal{A} \in  \mathbb{C}^{I_1 \times \cdots \times I_M\times I_1 \times \cdots \times I_N}$, \\ $\mathcal{B} \in  \mathbb{C}^{I_1 \times \cdots \times I_K\times I_1 \times \cdots \times I_K}$, $\mathcal{U} \in  \mathbb{C}^{I_1 \times \cdots \times I_M\times I_1 \times \cdots \times I_K}$ and $\mathcal{V} \in  \mathbb{C}^{I_1 \times \cdots \times I_K \times I_1 \times \cdots \times I_N}$, if following conditions are valid :
\begin{enumerate}\label{enu: MP inverse conditions X zero}
  \item $\mathcal{U} = \mathcal{Y}_1$, where $\mathcal{Y}_1$ is orthgonal to $\mathfrak{C}(\mathcal{A})$;
  \item $\mathcal{V}^{H} =  \mathcal{Y}_2$, where $\mathcal{Y}_2$ is orthgonal to $\mathfrak{C}(\mathcal{A}^{H})$;
 \item (1) $\mathcal{E}_2 \mathcal{B}^{\dagger} \mathcal{E}_1^{H}\mathcal{Y}_1\mathcal{B} = \mathcal{E}_2$, (2) $\mathcal{Y}_1 \mathcal{E}^{H}_1 \mathcal{Y}_1 =  \mathcal{Y}_1$;
 \item (1) $\mathcal{B} \mathcal{Y}_2^{H} \mathcal{E}_2\mathcal{B}^{\dagger}\mathcal{E}_1^{H} = \mathcal{E}_1^{H}$, (2) $\mathcal{E}_2 \mathcal{Y}^{H}_2 \mathcal{E}_2=  \mathcal{E}_2 $.
\end{enumerate}
Then the tensor  
%
\begin{eqnarray}\label{eq: S tensors X zero}
\mathcal{S} &=& \mathcal{A} + \mathcal{U} \star_K \mathcal{B} \star_K \mathcal{V} \nonumber \\
&=& \mathcal{A} + \mathcal{Y}_1  \star_K \mathcal{B} \star_K \mathcal{Y}_2^H, 
\end{eqnarray}
has the following Moore-Penrose generalized inverse identiy:
\begin{eqnarray}\label{eq: Sherman Morrison Woodbury identity for Moore-Penrose inverse tensors X zero}
\mathcal{S}^{\dagger} &=& \mathcal{A}^{\dagger} + \mathcal{E}_2 \star_K \mathcal{B}^{\dagger}  \star_K \mathcal{E}^{H}_1,
\end{eqnarray}
where $\mathcal{E}_i \define \mathcal{Y}_i ( \mathcal{Y}^{H}_i \mathcal{Y}_i )^{\dagger}$ for $i =1, 2$.

\end{corollary}
\textbf{Proof: }

The proof can be obtained by replacing tensors $\mathcal{X}_1$ and $\mathcal{X}_2$ as zero tensor $\mathcal{O}$ in the proof of Theorem~\ref{thm: Identity for Tensors with Moore-Penrose Inverse}.  $\hfill \Box$

If the tensor $\mathcal{A}$ is a Hermitian tensor, we can have following corollary. 
\begin{corollary}
\label{thm: Identity for Tensors with Moore-Penrose Inverse Hermitian}
Given tensors $\mathcal{A} \in  \mathbb{C}^{I_1 \times \cdots \times I_M\times I_1 \times \cdots \times I_M}$, \\
$\mathcal{B} \in \mathbb{C}^{I_1 \times \cdots \times I_K\times I_1 \times \cdots \times I_K}$, $\mathcal{U} \in  \mathbb{C}^{I_1 \times \cdots \times I_M\times I_1 \times \cdots \times I_K}$ and $\mathcal{V} \in  \mathbb{C}^{I_1 \times \cdots \times I_K \times I_1 \times \cdots \times I_M}$, if following conditions are valid :
\begin{enumerate}\label{enu: MP inverse conditions Hermitian}
  \item $\mathcal{A}^{H} = \mathcal{A}$, and $\mathcal{U}^{H} = \mathcal{V}$.
  \item $\mathcal{U} = \mathcal{X} + \mathcal{Y}$, where $\mathcal{X} \in \mathfrak{C}(\mathcal{A})$ and $\mathcal{Y}$ is orthgonal to $\mathfrak{C}(\mathcal{A})$;
  \item (1) $\mathcal{E} \mathcal{B}^{\dagger} \mathcal{E}^{H}\mathcal{Y}\mathcal{B} = \mathcal{E}$, (2) $\mathcal{X} \mathcal{E}^{H} \mathcal{Y}\mathcal{B} = \mathcal{X}\mathcal{B}$, (3) $\mathcal{Y} \mathcal{E}^{H} \mathcal{Y} =  \mathcal{Y}$;
 \item (1) $\mathcal{B} \mathcal{Y}^{H} \mathcal{E}\mathcal{B}^{\dagger}\mathcal{E}^{H} = \mathcal{E}^{H}$, (2)  $\mathcal{B} \mathcal{Y}^{H} \mathcal{E}\mathcal{X}^{H} = \mathcal{B}\mathcal{X}^{H}$, (3) $\mathcal{E} \mathcal{Y}^{H} \mathcal{E} = \mathcal{E}$.
\end{enumerate}
Then the tensor  
\begin{eqnarray}\label{eq: S tensors Hermitian}
\mathcal{S} &=& \mathcal{A} + \mathcal{U} \star_K \mathcal{B} \star_K \mathcal{V} \nonumber \\
&=& \mathcal{A} + (\mathcal{X} + \mathcal{Y} ) \star_K \mathcal{B} \star_K (\mathcal{X} + \mathcal{Y} )^H, 
\end{eqnarray}
has the following Moore-Penrose generalized inverse identiy:
\begin{eqnarray}\label{eq: Sherman Morrison Woodbury identity for Moore-Penrose inverse tensors Hermitian}
\mathcal{S}^{\dagger} &=& \mathcal{A}^{\dagger} - \mathcal{E} \star_K \mathcal{X}^{H} \star_N \mathcal{A}^{\dagger} - \mathcal{A}^{\dagger} \star_M \mathcal{X}\star_K \mathcal{E}^{H} \nonumber \\
& &+\mathcal{E} \star_K (\mathcal{B}^{\dagger} + \mathcal{X}^{H} \star_N \mathcal{A}^{\dagger} \star_M \mathcal{X}) \star_K \mathcal{E}^{H},
\end{eqnarray}
where $\mathcal{E}\define \mathcal{Y} ( \mathcal{Y}^{H} \mathcal{Y} )^{\dagger}$. 

\end{corollary}
\textbf{Proof: }
Because $\mathcal{A}^{H} = \mathcal{A}$, and $\mathcal{U}^{H} = \mathcal{V}$, the proof from 
Theorem~\ref{thm: Identity for Tensors with Moore-Penrose Inverse} can be applied here by removing those subscript indices, $1$ and $2$. $\hfill \Box$

\subsection{Illustrative Examples}\label{sec: Example}

In this section, we will provide two examples to demonstarte the validity of Theorem~\ref{thm: Identity for Tensors with Moore-Penrose Inverse} and Corollary~\ref{cor: Identity for Tensors with Moore-Penrose Inverse X zero}. We have to use following tensor equation in this section.
\begin{eqnarray}\label{AZB equal D}
\mathcal{A} \star_N \mathcal{Z} \star_M \mathcal{B} =  ( \mathcal{A} \otimes \mathcal{B}^{H} ) \star_{(N+M)} \mathcal{Z},
\end{eqnarray}
where $\otimes$ is the Kronecker product of tensors~\cite{sun2016moore}. 

Following example is provided to verify Corollary~\ref{cor: Identity for Tensors with Moore-Penrose Inverse X zero}.

\begin{example}\label{example 1}
Given tensor $\mathcal{A} \in \mathbb{R}^{2 \times 2 \times 2 \times 2}$
\begin{eqnarray}\label{eq: tensor A at example 1}
\mathcal{A} &=&
\left[
    \begin{array}{cc : cc}
       a_{11,11} & a_{12,11} & a_{11,12} & a_{12,12}  \\
       a_{21,11} & a_{22,11} & a_{21,12} & a_{22,12}  \\ \hdashline[2pt/2pt]
       a_{11,21} & a_{12,21} & a_{11,22} & a_{12,22}   \\
       a_{21,21} & a_{22,21} & a_{11,22} & a_{22,22}   \\
    \end{array}
\right]
=
\left[
    \begin{array}{cc : cc}
       1 & -1 & 0 & 0  \\
       0 &  0 & -1 & 0 \\ \hdashline[2pt/2pt]
       0 &  1 & 0 & 0  \\
       0 &  0 & 1 & 0  \\
    \end{array}
\right]
\end{eqnarray}
Then, the column tensor $\mathbf{a}_{11}$ of tensor $\mathcal{A}$ is $
\left[
    \begin{array}{cc}
       1 & -1  \\
       0 &  0 
    \end{array}
\right]$,  the column tensor $\mathbf{a}_{12}$ of tensor $\mathcal{A}$ is $
\left[
    \begin{array}{cc}
       0 & 0  \\
       -1 &  0 
    \end{array}
\right]$,  the column tensor  $\mathbf{a}_{21}$ of tensor $\mathcal{A}$ is $
\left[
    \begin{array}{cc}
       0 & 1  \\
       0 &  0 
    \end{array}
\right]$,  and the column tensor $\mathbf{a}_{22}$ of tensor $\mathcal{A}$ is $
\left[
    \begin{array}{cc}
       0 & 0  \\
       1 &  0 
    \end{array}
\right]$.  

If we take Hermition for the tensor $\mathcal{A}$, we have 
\begin{eqnarray}\label{eq: tensor A H at example 1}
\mathcal{A}^{H} &=&
\left[
    \begin{array}{cc : cc}
       a_{11,11} & a_{11,12} & a_{12,11} & a_{12,12}  \\
       a_{11,21} & a_{11,22} & a_{12,21} & a_{12,22}  \\ \hdashline[2pt/2pt]
       a_{21,11} & a_{21,12} & a_{22,11} & a_{22,12}   \\
       a_{21,21} & a_{21,22} & a_{22,21} & a_{22,22}   \\
    \end{array}
\right]
=
\left[
    \begin{array}{cc : cc}
       1 & 0 & -1 & 0  \\
       0 &  0 & 1 & 0 \\ \hdashline[2pt/2pt]
       0 &  -1 & 0 & 0  \\
       0 &  1  & 0 & 0  \\
    \end{array}
\right]
\end{eqnarray}
Then, the column tensor $\mathbf{a}_{11}$ of tensor $\mathcal{A}^{H}$ is $
\left[
    \begin{array}{cc}
       1 &  0  \\
       0 &  0 
    \end{array}
\right]$,  the column tensor $\mathbf{a}_{12}$ of tensor $\mathcal{A^{H}}$ is $
\left[
    \begin{array}{cc}
       -1 & 0  \\
       1 &  0 
    \end{array}
\right]$,  the column tensor  $\mathbf{a}_{21}$ of tensor $\mathcal{A}^{H}$ is $
\left[
    \begin{array}{cc}
       0 & -1  \\
       0 &  1 
    \end{array}
\right]$,  and the column tensor $\mathbf{a}_{22}$ of tensor $\mathcal{A}^{H}$ is $
\left[
    \begin{array}{cc}
       0 &  0  \\
       0 &  0 
    \end{array}
\right]$.

The tensor $\mathcal{B} \in \mathbb{R}^{1 \times 1 \times 1 \times 1}$ has only one entry with value $1$, the value in this tensor $\mathcal{B}$ is denoted as $1 \in \mathbb{R}^{1 \times 1 \times 1 \times 1}$ . The tensor $\mathcal{U} \in \mathbb{R}^{2 \times 2 \times 1 \times 1}$ is
\begin{eqnarray}\label{eq: tensor U at example 1}
\mathcal{U} &=&
\left[
    \begin{array}{cc}
       0 &  0 \\
       0 &  1 
    \end{array}
\right],
\end{eqnarray}
and 
the tensor $\mathcal{V} \in \mathbb{R}^{1 \times 1 \times 2 \times 2}$ is
\begin{eqnarray}\label{eq: tensor V at example 1}
\mathcal{V} &=&
\left[
    \begin{array}{cc}
       0 &  1 \\
       0 &  1 
    \end{array}
\right].  
\end{eqnarray}

Then, we have the tensor $\mathcal{S}$ expressed as 
\begin{eqnarray}\label{eq: tensor S at example 1}
\mathcal{S}&=&\mathcal{A} + \mathcal{U} \mathcal{B} \mathcal{V} \nonumber \\
&=& \left[
    \begin{array}{cc : cc}
       1 & -1 & 0 & 0  \\
       0 &  0 & -1 & 0 \\ \hdashline[2pt/2pt]
       0 &  1 & 0 & 0  \\
       0 &  0 & 1 & 0  \\
    \end{array}
\right] + \mathcal{U} \otimes \mathcal{V}^{H} \star_{4} \mathcal{B} \nonumber \\
&=&\left[
    \begin{array}{cc : cc}
       1 & -1 & 0 & 0  \\
       0 &  0 & -1 & 0 \\ \hdashline[2pt/2pt]
       0 &  1 & 0 & 0  \\
       0 &  0 & 1 & 0  \\
    \end{array}
\right] + \left[
    \begin{array}{cc : cc}
       0 &  0 & 0 & 0  \\
       0 &  0 & 0 & 1 \\ \hdashline[2pt/2pt]
       0 &  0 & 0 & 0  \\
       0 &  0 & 0 & 1  \\
    \end{array}
\right] = \left[
    \begin{array}{cc : cc}
       1 & -1 & 0 & 0  \\
       0 &  0 & -1 & 1 \\ \hdashline[2pt/2pt]
       0 &  1 & 0 & 0  \\
       0 &  0 & 1 & 1  \\
    \end{array}
\right].
\end{eqnarray}
The goal is to verify SHERMAN-MORRISON-WOODBURY identiy for the inverse of the tensor $\mathcal{S}$.

Because the tensor $\mathcal{A}$ is not invertible, the Moore-Penrose inverse of the tensor $\mathcal{A}$ becomes
\begin{eqnarray}
\mathcal{A}^{\dagger} &=&
\left[
    \begin{array}{cc : cc}
       1 & 0 & 0 & 0  \\
       1 & 0 & 1 & 0 \\ \hdashline[2pt/2pt]
       0 &  -\frac{1}{2} & 0 & 0  \\
       0 &  \frac{1}{2} & 0 & 0  \\
    \end{array}
\right].
\end{eqnarray}

Before applying Theorem~\ref{thm: Identity for Tensors with Moore-Penrose Inverse}, we have to decompose the tensors $\mathcal{U}$ and $\mathcal{V}^{H}$ according to the column-tensor spaces $\mathfrak{C}(\mathcal{A})$ and $\mathfrak{C}(\mathcal{A}^{H})$. They are decomposed as following:
\begin{eqnarray}\label{eq: decomp U example 1}
\mathcal{U} &=& \mathcal{Y}_1 \nonumber \\
&=& \left[
    \begin{array}{cc}
       0 &  0 \\
       0 &  1 
    \end{array}
\right];
\end{eqnarray}
and 
\begin{eqnarray}\label{eq: decomp V example 1}
\mathcal{V} &=& \mathcal{Y}^H_2 \nonumber \\
&=& \left[
    \begin{array}{cc}
       0 &  1 \\
       0 &  1 
    \end{array}
\right].
\end{eqnarray}
Therefore, the tensors $\mathcal{U}$ and $\mathcal{V}$ are at orthogonal spaces of $\mathfrak{C}(\mathcal{A})$ and $\mathfrak{C}(\mathcal{A}^{H})$, respectively.

Since we define $\mathcal{E}_i \define \mathcal{Y}_i ( \mathcal{Y}^{H}_i \mathcal{Y}_i )^{\dagger}$ for $i =1, 2$, the tensors $\mathcal{E}_i$ can be evaluated as 
\begin{eqnarray}\label{eq: E 1 example 1}
\mathcal{E}_1 &=&\mathcal{Y}_1 ( \mathcal{Y}^{H}_1 \mathcal{Y}_1 )^{\dagger} \nonumber \\
&=&
\left[
    \begin{array}{cc}
       0 &  0 \\
       0 &  1 
    \end{array}
\right] \star_2 (1 \in \mathbb{R}^{1 \times 1 \times 1 \times 1}) \nonumber \\
&=&
\left[
    \begin{array}{cc}
       0 &  0 \\
       0 &  1 
    \end{array}
\right] \in \mathbb{R}^{2 \times 2 \times 1 \times 1};
\end{eqnarray}
and
\begin{eqnarray}\label{eq: E 2 example 1}
\mathcal{E}_2 &=&\mathcal{Y}_2 ( \mathcal{Y}^{H}_2 \mathcal{Y}_2 )^{\dagger} \nonumber \\
&=&
\left[
    \begin{array}{cc}
       0 &  1 \\
       0 &  1 
    \end{array}
\right] \star_2 (\frac{1}{2} \in \mathbb{R}^{1 \times 1 \times 1 \times 1} ) \nonumber \\
&=&
\left[
    \begin{array}{cc}
       0 &  \frac{1}{2} \\
       0 &  \frac{1}{2}
    \end{array}
\right] \in \mathbb{R}^{2 \times 2 \times 1 \times 1},
\end{eqnarray}
where $\frac{1}{2} \in \mathbb{R}^{1 \times 1 \times 1 \times 1}$ is a single entry tensor with value $\frac{1}{2}$ with tensor dimension $1 \times 1 \times 1 \times 1$ (order 4).

We are ready to evalute following terms $\mathcal{E}\mathcal{X}^{H}_2 \mathcal{A}^{\dagger}$, $\mathcal{A}^{\dagger} \mathcal{X}_1 \mathcal{E}_1^{H}$, and \\ $\mathcal{E}_2 (\mathcal{B}^{\dagger} + \mathcal{X}^{H}_2 \mathcal{A}^{\dagger} \mathcal{X}_1 ) \mathcal{E}^{H}_1$. But tensors $\mathcal{E}\mathcal{X}^{H}_2 \mathcal{A}^{\dagger}$ and $\mathcal{A}^{\dagger} \mathcal{X}_1 \mathcal{E}_1^{H}$ are zero tensors since $\mathcal{X}_1$ and  $\mathcal{X}_2$ are zero tensors. Because we also have following:
\begin{eqnarray}\label{X 2 A dagger X 1 example 1}
 \mathcal{X}^{H}_2 \mathcal{A}^{\dagger} \mathcal{X}_1 &=&  0 \in \mathbb{R}^{1 \times 1 \times 1 \times 1},
\end{eqnarray}
we have
\begin{eqnarray}\label{E 2 B dagger X 2 example 1}
\mathcal{E}_2 (\mathcal{B}^{\dagger} + \mathcal{X}^{H}_2 \mathcal{A}^{\dagger} \mathcal{X}_1 ) \mathcal{E}^{H}_1 &=& \mathcal{E}_2 \star_2 ( 1 \in \mathbb{R}^{1 \times 1 \times 1 \times 1} ) \star_2 \mathcal{E}^{H}_1 \nonumber \\
&=&  \left[
    \begin{array}{cc : cc}
       0 & 0 & 0 & 0  \\
       0 & 0 & 0 & 0 \\ \hdashline[2pt/2pt]
       0 &  0 & 0 & \frac{1}{2}  \\
       0 &  0 & 0 & \frac{1}{2} 
    \end{array}
\right]
\end{eqnarray}

Finally, from the Eq.~\eqref{E 2 B dagger X 2 example 1}, we have 
\begin{eqnarray}
\mathcal{S}^{\dagger} &=& \mathcal{A}^{\dagger} + \mathcal{E}_2 (\mathcal{B}^{\dagger} + \mathcal{X}^{H}_2 \mathcal{A}^{\dagger} \mathcal{X}_1 ) \mathcal{E}^{H}_1 \nonumber \\
&=& \left[
    \begin{array}{cc : cc}
       1 & 0 & 0 & 0  \\
       1 & 0 & 1 & 0 \\ \hdashline[2pt/2pt]
       0 &  -\frac{1}{2} & 0 & \frac{1}{2}  \\
       0 &  \frac{1}{2} & 0 & \frac{1}{2} 
    \end{array}
\right],
\end{eqnarray}
which is the inverse of the tensor $\mathcal{S}$. 

\end{example}

Following example will be more complicated by considering the situations that projected column-tensor parts of $\mathcal{U}$ and $\mathcal{V}$ are non-zero tensors. See Theorem~\ref{thm: Identity for Tensors with Moore-Penrose Inverse}.

\begin{example}\label{example 2}
Given same tensors $\mathcal{A} \in \mathbb{R}^{2 \times 2 \times 2 \times 2}$ and $\mathcal{B} \in \mathbb{R}^{1 \times 1 \times 1 \times 1}$ as Example~\ref{example 1}, the tensor $\mathcal{U} \in \mathbb{R}^{2 \times 2 \times 1 \times 1}$ is
\begin{eqnarray}\label{eq: tensor U at example 2}
\mathcal{U} &=&
\left[
    \begin{array}{cc}
       0 &  1 \\
       0 &  1 
    \end{array}
\right],
\end{eqnarray}
and 
the tensor $\mathcal{V} \in \mathbb{R}^{1 \times 1 \times 2 \times 2}$ is
\begin{eqnarray}\label{eq: tensor V at example 2}
\mathcal{V} &=&
\left[
    \begin{array}{cc}
       0 &  0 \\
       0 &  2 
    \end{array}
\right].  
\end{eqnarray}

Then, we have the tensor $\mathcal{S}$ expressed as 
\begin{eqnarray}\label{eq: tensor S at example 2}
\mathcal{S}&=&\mathcal{A} + \mathcal{U} \mathcal{B} \mathcal{V} \nonumber \\
&=& \left[
    \begin{array}{cc : cc}
       1 & -1 & 0 & 0  \\
       0 &  0 & -1 & 0 \\ \hdashline[2pt/2pt]
       0 &  1 & 0 & 0  \\
       0 &  0 & 1 & 0  \\
    \end{array}
\right] + \mathcal{U} \otimes \mathcal{V}^{H} \star_{4} \mathcal{B} \nonumber \\
&=&\left[
    \begin{array}{cc : cc}
       1 & -1 & 0 & 0  \\
       0 &  0 & -1 & 0 \\ \hdashline[2pt/2pt]
       0 &  1 & 0 & 0  \\
       0 &  0 & 1 & 0  \\
    \end{array}
\right] + \left[
    \begin{array}{cc : cc}
       0 &  0 & 0 & 0  \\
       0 &  0 & 0 & 0 \\ \hdashline[2pt/2pt]
       0 &  0 & 0 & 2  \\
       0 &  0 & 0 & 2  \\
    \end{array}
\right] = \left[
    \begin{array}{cc : cc}
       1 & -1 & 0 & 0  \\
       0 &  0 & -1 & 0 \\ \hdashline[2pt/2pt]
       0 &  1 & 0 & 2  \\
       0 &  0 & 1 & 2  \\
    \end{array}
\right].
\end{eqnarray}
We wish to show SHERMAN-MORRISON-WOODBURY identiy for the inverse of the tensor $\mathcal{S}$.

Before applying Theorem~\ref{thm: Identity for Tensors with Moore-Penrose Inverse}, we have to decompose the tensors $\mathcal{U}$ and $\mathcal{V}^{H}$ according to the column-tensor spaces $\mathfrak{C}(\mathcal{A})$ and $\mathfrak{C}(\mathcal{A}^{H})$. They are decomposed as following:
\begin{eqnarray}\label{eq: decomp U example 2}
\mathcal{U} &=& \mathcal{X}_1 + \mathcal{Y}_1 \nonumber \\
&=&
\left[
    \begin{array}{cc}
       0 &  1 \\
       0 &  0 
    \end{array}
\right] + \left[
    \begin{array}{cc}
       0 &  0 \\
       0 &  1 
    \end{array}
\right],
\end{eqnarray}
and 
\begin{eqnarray}\label{eq: decomp V example 2}
\mathcal{V} &=& \mathcal{X}^H_2 + \mathcal{Y}^H_2 \nonumber \\
&=&
\left[
    \begin{array}{cc}
       0 &  -1 \\
       0 &  1 
    \end{array}
\right] + \left[
    \begin{array}{cc}
       0 &  1 \\
       0 &  1 
    \end{array}
\right].
\end{eqnarray}
Under this decomposition, the subtensors $\mathcal{X}_1$ and  $\mathcal{X}_2$ are in the column-tensor spaces $\mathfrak{C}(\mathcal{A})$ and $\mathfrak{C}(\mathcal{A}^{H})$, respectively. Moreover, the subtensors $\mathcal{Y}_1$ and $\mathcal{Y}_2$ are orthgonal to the column-tensor spaces $\mathfrak{C}(\mathcal{A})$ and $\mathfrak{C}(\mathcal{A}^{H})$, respectively. Since we define $\mathcal{E}_i \define \mathcal{Y}_i ( \mathcal{Y}^{H}_i \mathcal{Y}_i )^{\dagger}$ for $i =1, 2$, the tensors $\mathcal{E}_i$ are evaluated at Eqs.~\eqref{eq: E 1 example 1} and~\eqref{eq: E 2 example 1} since $\mathcal{Y}_i$ are same with the previous example.

We are ready to evalute following terms $\mathcal{E}\mathcal{X}^{H}_2 \mathcal{A}^{\dagger}$, $\mathcal{A}^{\dagger} \mathcal{X}_1 \mathcal{E}_1^{H}$, and \\ $\mathcal{E}_2 (\mathcal{B}^{\dagger} + \mathcal{X}^{H}_2 \mathcal{A}^{\dagger} \mathcal{X}_1 ) \mathcal{E}^{H}_1$.

\begin{eqnarray}\label{A V 1 C 1}
\mathcal{A}^{\dagger} \mathcal{X}_1 \mathcal{E}^{H}_1 &=& \left[
    \begin{array}{cc : cc}
       1 & 0 & 0 & 0  \\
       1 & 0 & 1 & 0 \\ \hdashline[2pt/2pt]
       0 &  -\frac{1}{2} & 0 & 0  \\
       0 &  \frac{1}{2} & 0 & 0  \\
    \end{array}
\right] \star_2 \left[
    \begin{array}{cc}
       0 & 1 \\
       0 & 0
    \end{array}
\right] \otimes \left[
    \begin{array}{cc}
       0 & 0 \\
       0 &  1
    \end{array}
\right] \nonumber \\
&=& \left[
    \begin{array}{cc : cc}
       1 & 0 & 0 & 0  \\
       1 & 0 & 1 & 0 \\ \hdashline[2pt/2pt]
       0 &  -\frac{1}{2} & 0 & 0  \\
       0 &  \frac{1}{2} & 0 & 0  \\
    \end{array}
\right] \star_2 \left[
    \begin{array}{cc : cc}
       0 & 0 & 0 & 0  \\
       0 & 0 & 0 & 0 \\ \hdashline[2pt/2pt]
       0 & 0 & 0 & 1  \\
       0 & 0 & 0 & 0  \\
    \end{array}
\right] = \left[
    \begin{array}{cc : cc}
       0 & 0 & 0 & 0  \\
       0 & 0 & 0 & 0 \\ \hdashline[2pt/2pt]
       0 & 0 & 0 & 0  \\
       0 & 0 & 1 & 0  \\
    \end{array}
\right]. 
\end{eqnarray}

\begin{eqnarray}\label{C 2 V 2 A}
\mathcal{E}_2 \mathcal{X}^{H}_2\mathcal{A}^{\dagger} &=& 
 \left[
    \begin{array}{cc}
       0 &  \frac{1}{2} \\
       0 &  \frac{1}{2}
    \end{array}
\right] \otimes \left[
    \begin{array}{cc}
       0 &  -1 \\
       0 &   1
    \end{array}
\right] \star_2 \left[
    \begin{array}{cc : cc}
       1 & 0 & 0 & 0  \\
       1 & 0 & 1 & 0 \\ \hdashline[2pt/2pt]
       0 &  -\frac{1}{2} & 0 & 0  \\
       0 &  \frac{1}{2} & 0 & 0  \\
    \end{array}
\right]  \nonumber \\
&=&  \left[
    \begin{array}{cc : cc}
       0 & 0 & 0 & -\frac{1}{2}  \\
       0 & 0 & 0 & -\frac{1}{2} \\ \hdashline[2pt/2pt]
       0 & 0 & 0 & \frac{1}{2}  \\
       0 & 0 & 0 & \frac{1}{2}  \\
    \end{array}
\right] \star_2 \left[
    \begin{array}{cc : cc}
       1 & 0 & 0 & 0  \\
       1 & 0 & 1 & 0 \\ \hdashline[2pt/2pt]
       0 &  -\frac{1}{2} & 0 & 0  \\
       0 &  \frac{1}{2} & 0 & 0  \\
    \end{array}
\right] = \left[
    \begin{array}{cc : cc}
       0 & 0 & 0 & 0  \\
       0 & 0 & 0 & 0 \\ \hdashline[2pt/2pt]
       0 & \frac{1}{2} & 0 & 0  \\
       0 & \frac{1}{2} & 0 & 0  \\
    \end{array}
\right]. 
\end{eqnarray}

Since we have following:
\begin{eqnarray}\label{X 2 A dagger X 1 example 2}
 \mathcal{X}^{H}_2 \mathcal{A}^{\dagger} \mathcal{X}_1 &=& \left[
    \begin{array}{cc}
       0 &  -1 \\
       0 &   1
    \end{array}
\right] \star_2 \left[
    \begin{array}{cc : cc}
       1 & 0 & 0 & 0  \\
       1 & 0 & 1 & 0 \\ \hdashline[2pt/2pt]
       0 &  -\frac{1}{2} & 0 & 0  \\
       0 &  \frac{1}{2} & 0 & 0  \\
    \end{array}
\right] \star_2
\left[
    \begin{array}{cc}
       0 &  1 \\
       0 &  0
    \end{array}
\right] \nonumber \\
&=& \left[
    \begin{array}{cc}
       0 &  -1 \\
       0 &   1
    \end{array}
\right] \star_2 \left[
    \begin{array}{cc}
       0 &  0 \\
       1 &  0
    \end{array}
\right] = 0 \in \mathbb{R}^{1 \times 1 \times 1 \times 1},
\end{eqnarray}
we have
\begin{eqnarray}\label{E 2 B dagger X 2 example 2}
\mathcal{E}_2 (\mathcal{B}^{\dagger} + \mathcal{X}^{H}_2 \mathcal{A}^{\dagger} \mathcal{X}_1 ) \mathcal{E}^{H}_1 &=& \mathcal{E}_2 \star_2 ( 1 \in \mathbb{R}^{1 \times 1 \times 1 \times 1} ) \star_2 \mathcal{E}^{H}_1 \nonumber \\
&=&   \left[
    \begin{array}{cc : cc}
       0 & 0 & 0 & 0  \\
       0 & 0 & 0 & 0 \\ \hdashline[2pt/2pt]
       0 &  0 & 0 & \frac{1}{2}  \\
       0 &  0 & 0 & \frac{1}{2} 
    \end{array}
\right]
\end{eqnarray}

Finally, from the Eqs.~\eqref{A V 1 C 1},~\eqref{C 2 V 2 A},~\eqref{E 2 B dagger X 2 example 2}, we have 
\begin{eqnarray}
\mathcal{S}^{\dagger} &=& \mathcal{A}^{\dagger} - \mathcal{A}^{\dagger} \mathcal{X}_1 \mathcal{E}^{H}_1 - \mathcal{E}_2 \mathcal{X}^{H}_2\mathcal{A}^{\dagger} + \mathcal{E}_2 (\mathcal{B}^{\dagger} + \mathcal{X}^{H}_2 \mathcal{A}^{\dagger} \mathcal{X}_1 ) \mathcal{E}^{H}_1 \nonumber \\
&=& \left[
    \begin{array}{cc : cc}
       1 & 0 & 0 & 0  \\
       1 & 0 & 1 & 0 \\ \hdashline[2pt/2pt]
       0 &  -1 & 0 & \frac{1}{2}  \\
       0 &  0 & -1 & \frac{1}{2} 
    \end{array}
\right],
\end{eqnarray}
which is the inverse of the tensor $\mathcal{S}$. 

\end{example}

\section{Application: Sensitivity Analysis for Multilinear Systems}\label{sec:Sensitivity Analysis for Multilinear Systems}

In this section, we will apply the results obtained in Section~\ref{sec: Identity for Tensors with Moore-Penrose Inverse} to perform sensitivity analysis for a multilinear system of equations, i.e. $\mathcal{A}\mathcal{X} = \mathcal{D}$,  by deriving the normalized upper bound for the error in the solution when coefficient tensors are perturbed in Section~\ref{sec:Sensitivity Analysis}. In Section~\ref{sec:Numerical Evaluation}, we investigate the effects of perturbation values $\epsilon_A, \epsilon_D$ to the normalized solution error $\frac{\left\Vert \mathcal{Y} - \mathcal{X}  \right\Vert } {\left\Vert \mathcal{X} \right\Vert}$, denoted as $E_n$. All norms discussed in this paper are based on the Frobenius norm definition. 

\subsection{Sensitivity Analysis}\label{sec:Sensitivity Analysis}

Serveral preparation lemmas will be given before presenting our results asscoaited to sensitivity analysis for multilinear systems.

\begin{lemma}\label{Submultiplicative of Frobenius Norm}
Let $\mathcal{A} \in \mathbb{C}^{I_1 \times \cdots \times I_M \times J_1 \times \cdots \times J_N}$
and $\mathcal{B} \in \mathbb{C}^{J_1 \times \cdots \times J_N \times K_1 \times \cdots \times K_L}$, we 
have following inequality for Frobenius norm of tensors: 
\begin{eqnarray}
\left\Vert \mathcal{A} \star_N \mathcal{B} \right\Vert \leq \left\Vert \mathcal{A} \right\Vert \left\Vert  \mathcal{B} \right\Vert.
\end{eqnarray}
\end{lemma}
\begin{proof}
Because
\begin{eqnarray}
\left\Vert \mathcal{A} \star_N \mathcal{B} \right\Vert^2 &=& \sum_{i_1, \cdots, i_M,k_1 \times \cdots \times k_L} |\sum\limits_{j_1, \cdots, j_N} a_{i_1, \cdots, i_M, j_1, \cdots,j_N}b_{j_1, \cdots, j_N, k_1, \cdots,k_L}|^2 \nonumber \\
&\leq&  \sum_{i_1, \cdots, i_M, k_1 \times \cdots \times k_L} \left[ \left(\sum\limits_{j_1, \cdots, j_N} | a_{i_1, \cdots, i_M, j_1, \cdots,j_N} |^2 \right) \right. \times \nonumber \\
&  & \left. \left(\sum\limits_{j_1, \cdots, j_N} | b_{j_1, \cdots, j_N, k_1, \cdots,k_L} |^2 \right) \right] \nonumber \\
&=&  \left( \sum\limits_{i_1, \cdots, i_M} \left(\sum\limits_{j_1, \cdots, j_N} | a_{i_1, \cdots, i_M, j_1, \cdots,j_N} |^2 \right) \right) \times \nonumber \\
&  &\left(\sum_{j_1, \cdots, j_N} \left(\sum\limits_{ k_1 \times \cdots \times k_L} | b_{j_1, \cdots, j_N, k_1, \cdots,k_L} |^2 \right) \right) =  \left\Vert \mathcal{A} \right\Vert^2\left\Vert \mathcal{B} \right\Vert^2
\end{eqnarray}
where the inequality is based on Cauchy–Schwarz inequality. By taking square root of both sides, the lemma is established. 
\end{proof}

\begin{lemma}\label{Triangle Inequality of Frobenius Norm}
Let $\mathcal{A} \in \mathbb{C}^{I_1 \times \cdots \times I_M \times J_1 \times \cdots \times J_N}$
and $\mathcal{B} \in \mathbb{C}^{I_1 \times \cdots \times I_M \times J_1 \times \cdots \times J_N}$, we 
have following inequality for Frobenius norm of tensors: 
\begin{eqnarray}
\left\Vert \mathcal{A} + \mathcal{B} \right\Vert \leq \left\Vert \mathcal{A} \right\Vert + \left\Vert  \mathcal{B} \right\Vert.
\end{eqnarray}
\end{lemma}
\begin{proof}
Because
\begin{eqnarray}
\left\Vert \mathcal{A} + \mathcal{B} \right\Vert ^2 &\leq&  \left\Vert \mathcal{A} \right\Vert^{2} + 
 \left\Vert \mathcal{B} \right\Vert^{2} + 
2 \sum\limits_{i_1, \cdots, i_M, j_1, \cdots,j_N} |a_{i_1, \cdots, i_M, j_1, \cdots,j_N} |  |b_{i_1, \cdots, i_M, j_1, \cdots,j_N} |  \nonumber \\
&\stackrel{1}{\leq}&  \left\Vert \mathcal{A} \right\Vert^{2} + 
 \left\Vert \mathcal{B} \right\Vert^{2} + 2 \left\Vert \mathcal{A} \right\Vert \left\Vert \mathcal{B} \right\Vert
 \nonumber \\ 
&=&(  \left\Vert \mathcal{A} \right\Vert  + \left\Vert \mathcal{B} \right\Vert   )^{2}
\end{eqnarray}
where the ineqsuality $\stackrel{1}{\leq}$ is based on Cauchy–Schwarz inequality. By taking square root of both sides, the lemma is established. 
\end{proof}

Given a multilinear system of equations, the exact solution expressed by tensor inverse or Moore-Penrose inverse is given by following Theorem. The proof can be found at~\cite{behera2017further}. 

\begin{theorem}\label{Theorem: AX equal D Moore}
For given tensors $\mathcal{A} \in \mathbb{C}^{K_1 \times \cdots \times K_P \times I_1 \times \cdots \times I_M}$, \nonumber \\
$\mathcal{D} \in \mathbb{C}^{K_1 \times \cdots \times K_P \times L_1 \times \cdots \times L_Q}$, the tensor equation
\begin{eqnarray}\label{eq: Theorem AX equal D Moore}
\mathcal{A} \star_M \mathcal{X} = \mathcal{D},
\end{eqnarray}
has a solution if and only if $\mathcal{A} \star_M \mathcal{A}^{\dagger} \star_P \mathcal{D} =  \mathcal{D}$.
The solution can be expressed as 
\begin{eqnarray}\label{eq: Theorem AX equal D Standard sol Moore}
\mathcal{X} = \mathcal{A}^{\dagger} \star_P \mathcal{D} + (\mathcal{I} - \mathcal{A}^{\dagger} \star_P \mathcal{A})\star_M \mathcal{U},
\end{eqnarray}
where $\mathcal{I}$ is the identiy tensor in $\mathbb{C}^{ I_1 \times \cdots \times I_M \times  I_1 \times \cdots \times I_M}$ and  $\mathcal{U}$ is an arbitrary tensor in $\mathbb{C}^{I_1 \times \cdots \times I_M \times L_1 \times \cdots \times L_Q}$.

If the tensor $\mathcal{A}$ is invertible, then the Eq.~\eqref{eq: Theorem AX equal D Standard sol Moore} can be further reduced as 
\begin{eqnarray}\label{eq: Theorem AX equal D Standard sol}
\mathcal{X} = \mathcal{A}^{-1} \star_P \mathcal{D}.
\end{eqnarray}
\end{theorem}

We are ready to present our theorem about sensitvity analysis for solution of a multilinear system.

\begin{theorem}\label{Theorem: sensitivity of AX equal D}
The original multilinear system of equations is 
\begin{eqnarray}\label{eq: AX equal D original}
\mathcal{A}\star_M \mathcal{X} = \mathcal{D} 
\end{eqnarray}
where $\mathcal{A} \in \mathbb{C}^{K_1 \times \cdots \times K_P \times I_1 \times \cdots \times I_M}$, $\mathcal{D} \in \mathbb{C}^{K_1 \times \cdots \times K_P \times L_1 \times \cdots \times L_Q}$, 
and \\
$\mathcal{O} \neq \mathcal{D} \in \mathbb{C}^{K_1 \times \cdots \times K_P \times L_1 \times \cdots \times L_Q}$. The perturbed system can be expressed as 
\begin{eqnarray}\label{eq: AX equal D with pert}
(\mathcal{A} + \delta \mathcal{A}) \star_M \mathcal{Y} = (\mathcal{D} + \delta \mathcal{D}), 
\end{eqnarray}
where $\delta \mathcal{A} \in \mathbb{C}^{K_1 \times \cdots \times K_P \times I_1 \times \cdots \times I_M}$
and $\delta \mathcal{D} \in \mathbb{C}^{K_1 \times \cdots \times K_P \times L_1 \times \cdots \times L_Q}$. If 
the tensor $\delta \mathcal{A}$ is decomposed as (for example, by SVD decomposition when $\delta \mathcal{A}$ is a square tensor, see~\cite{sun2016moore}) 
\begin{eqnarray}\label{delta A decomp}
\delta \mathcal{A} &=& \mathcal{U} \mathcal{B} \mathcal{V} \nonumber \\
&=&  ( \mathcal{X}_1 + \mathcal{Y}_1 ) \mathcal{B} ( \mathcal{X}_2 + \mathcal{Y}_2)^H,
\end{eqnarray}
where $\mathcal{X}_1 \in \mathfrak{C}(\mathcal{A})$, $\mathcal{Y}_1$ is orthgonal to $\mathfrak{C}(\mathcal{A})$, $\mathcal{X}_2 \in \mathfrak{C}(\mathcal{A}^{H})$ and $\mathcal{Y}_2$ is orthgonal to $\mathfrak{C}(\mathcal{A}^{H})$.

We further assume that $\left\Vert \mathcal{X}_i \right\Vert \leq \epsilon_{A} \left\Vert \mathcal{A} \right\Vert$ for $1 \leq i \leq 2$, $\left\Vert \mathcal{E}_i \right\Vert \leq \epsilon_{A} \left\Vert \mathcal{A} \right\Vert$ for $1 \leq i \leq 2$ (Recall $\mathcal{E}_i \define \mathcal{Y}_i ( \mathcal{Y}^{H}_i \mathcal{Y}_i )^{\dagger}$) and   
$\left\Vert \delta \mathcal{D} \right\Vert \leq \epsilon_{D} \left\Vert \mathcal{B} \right\Vert$, then
\begin{eqnarray}
\frac{\left\Vert \mathcal{Y} - \mathcal{X} \right\Vert}{\left\Vert \mathcal{X} \right\Vert} &\leq& (1 + \epsilon_D) \left\Vert \mathcal{A} \right\Vert^3 (2 \epsilon^2_{A} \left\Vert \mathcal{A}^{\dagger} \right\Vert + \epsilon^3_{A} \left\Vert \mathcal{A} \right\Vert + \epsilon^4_{A} \left\Vert \mathcal{A} \right\Vert^2 \left\Vert \mathcal{A}^{\dagger} \right\Vert) + \nonumber \\
& & \epsilon_D  \left\Vert \mathcal{A} \right\Vert \left\Vert \mathcal{A}^{\dagger} \right\Vert.
\end{eqnarray}
\end{theorem}
\textbf{Proof:} 

From Theorem~\ref{Theorem: AX equal D Moore} and the Eq.~\ref{eq: Theorem AX equal D Standard sol Moore}, the solution for the Eq.~\eqref{eq: AX equal D original} is 
\begin{eqnarray}\label{eq: pert proof 1}
\mathcal{X} = \mathcal{A}^{\dagger} \star_P \mathcal{D} + (\mathcal{I} - \mathcal{A}^{\dagger} \star_P \mathcal{A})\star_M \mathcal{U},
\end{eqnarray}
and, similarly, the solution for the Eq.~\eqref{eq: AX equal D with pert} is 
\begin{eqnarray}\label{eq: pert proof 2}
\mathcal{Y} = (\mathcal{A} + \delta \mathcal{A})^{\dagger} \star_P (\mathcal{D} + \delta \mathcal{D})+ (\mathcal{I} -  (\mathcal{A} + \delta \mathcal{A})^{\dagger} \star_P  (\mathcal{A} + \delta \mathcal{A})\star_M \mathcal{U}.
\end{eqnarray}
Since the tensor $\mathcal{U}$ can be chosen arbitraryly, we can set $\mathcal{U}$ as zero tensor and we have
\begin{eqnarray}\label{eq: pert proof 3}
\mathcal{Y} - \mathcal{X} &=& (\mathcal{A} + \delta \mathcal{A} )^{\dagger}  (\mathcal{D} + \delta \mathcal{D}) - \mathcal{A}^{\dagger} \mathcal{D} \nonumber \\
&=& \left[ \mathcal{A} + ( \mathcal{X}_1 + \mathcal{Y}_1 ) \mathcal{B} ( \mathcal{X}_2 + \mathcal{Y}_2)^H \right]^{\dagger}  (\mathcal{D} + \delta \mathcal{D}) - \mathcal{A}^{\dagger} \mathcal{D} \nonumber \\
&\stackrel{1}{=}&( \mathcal{A}^{\dagger} - \mathcal{E}_2 \mathcal{X}^{H}_2\mathcal{A}^{\dagger} -  \mathcal{A}^{\dagger}\mathcal{X}_1 \mathcal{E}^{H}_1 + \mathcal{E}_2 (\mathcal{B}^{\dagger} + \mathcal{X}^{H}_2 \mathcal{A}^{\dagger} \mathcal{X}_1 ) \mathcal{E}^{H}_1) (\mathcal{D} + \delta \mathcal{D}) - \mathcal{A}^{\dagger} \mathcal{D} \nonumber \\
&=& \mathcal{E}_2 \mathcal{X}^{H}_2\mathcal{A}^{\dagger}\mathcal{D} -  \mathcal{A}^{\dagger}\mathcal{X}_1 \mathcal{E}^{H}_1\mathcal{D} + \mathcal{E}_2 (\mathcal{B}^{\dagger} + \mathcal{X}^{H}_2 \mathcal{A}^{\dagger} \mathcal{X}_1 ) \mathcal{E}^{H}_1 \mathcal{D} + \mathcal{A}^{\dagger} \delta \mathcal{D} - \nonumber \\
& & \mathcal{E}_2 \mathcal{X}^{H}_2\mathcal{A}^{\dagger} \delta \mathcal{D} -  \mathcal{A}^{\dagger}\mathcal{X}_1 \mathcal{E}^{H}_1 \delta \mathcal{D}+ \mathcal{E}_2 (\mathcal{B}^{\dagger} + \mathcal{X}^{H}_2 \mathcal{A}^{\dagger} \mathcal{X}_1 ) \mathcal{E}^{H}_1 \delta \mathcal{D},
\end{eqnarray}
where we apply Theorem~\ref{thm: Identity for Tensors with Moore-Penrose Inverse} at $\stackrel{1}{=}$. If we take Frobenius norm at both sides of Eq.~\eqref{eq: pert proof 3}, we get 
\begin{eqnarray}\label{eq: pert proof 4}
\left\Vert \mathcal{Y} - \mathcal{X} \right\Vert &=& \left\Vert \mathcal{E}_2 \mathcal{X}^{H}_2\mathcal{A}^{\dagger}\mathcal{D} -  \mathcal{A}^{\dagger}\mathcal{X}_1 \mathcal{E}^{H}_1\mathcal{D} + \mathcal{E}_2 (\mathcal{B}^{\dagger} + \mathcal{X}^{H}_2 \mathcal{A}^{\dagger} \mathcal{X}_1 ) \mathcal{E}^{H}_1 \mathcal{D} + \mathcal{A}^{\dagger} \delta \mathcal{D} \right. - \nonumber \\
& & \left. \mathcal{E}_2 \mathcal{X}^{H}_2\mathcal{A}^{\dagger} \delta \mathcal{D} -  \mathcal{A}^{\dagger}\mathcal{X}_1 \mathcal{E}^{H}_1 \delta \mathcal{D}+ \mathcal{E}_2 (\mathcal{B}^{\dagger} + \mathcal{X}^{H}_2 \mathcal{A}^{\dagger} \mathcal{X}_1 ) \mathcal{E}^{H}_1 \delta \mathcal{D} \right\Vert \nonumber \\
&\stackrel{2}{\leq}& 
\left\Vert \mathcal{E}_2   \right\Vert \left\Vert \mathcal{X}^{H}_2  \right\Vert \left\Vert \mathcal{A}^{\dagger}  \right\Vert \left\Vert\mathcal{D}  \right\Vert + \left\Vert  \mathcal{A}^{\dagger}  \right\Vert \left\Vert \mathcal{X}_1 \right\Vert  \left\Vert \mathcal{E}^{H}_1  \right\Vert \left\Vert \mathcal{D} \right\Vert  + \left\Vert \mathcal{E}_2  \right\Vert \left\Vert \mathcal{B}^{\dagger} \right\Vert  \left\Vert  \mathcal{E}^{H}_1 \right\Vert  \left\Vert   \mathcal{D}  \right\Vert  + \nonumber \\
& & \left\Vert \mathcal{X}^{H}_2  \right\Vert \left\Vert \mathcal{A}^{\dagger} \right\Vert  \left\Vert \mathcal{X}_1 \right\Vert  \left\Vert \mathcal{E}^{H}_1  \right\Vert \left\Vert \mathcal{D}  \right\Vert + 
\left\Vert \mathcal{A}^{\dagger}  \right\Vert\left\Vert \delta \mathcal{D}  \right\Vert+ \left\Vert \mathcal{E}_2 \right\Vert \left\Vert \mathcal{X}^{H}_2 \right\Vert \left\Vert \mathcal{A}^{\dagger} \right\Vert \left\Vert \delta \mathcal{D} \right\Vert +
\nonumber \\
& &  \left\Vert \mathcal{A}^{\dagger} \right\Vert \left\Vert \mathcal{X}_1 \right\Vert \left\Vert \mathcal{E}^{H}_1 \right\Vert \left\Vert \delta \mathcal{D} \right\Vert+ \left\Vert \mathcal{E}_2 \right\Vert \left\Vert \mathcal{B}^{\dagger} \right\Vert \left\Vert \mathcal{E}^{H}_1 \right\Vert \left\Vert  \delta \mathcal{D} \right\Vert + \nonumber \\
&  &\left\Vert \mathcal{X}^{H}_2 \right\Vert \left\Vert \mathcal{A}^{\dagger} \right\Vert \left\Vert  \mathcal{X}_1 \right\Vert \left\Vert \mathcal{E}^{H}_1 \right\Vert \left\Vert \delta \mathcal{D} \right\Vert \end{eqnarray}
where we apply Lemma~\ref{Submultiplicative of Frobenius Norm} and Lemma~\ref{Triangle Inequality of Frobenius Norm} to the
inequality $\stackrel{2}{\leq}$. Because we have that $\left\Vert \mathcal{X}_i \right\Vert \leq \epsilon_{A} \left\Vert \mathcal{A} \right\Vert$ for $1 \leq i \leq 2$, $\left\Vert \mathcal{E}_i \right\Vert \leq \epsilon_{A} \left\Vert \mathcal{A} \right\Vert$ for $1 \leq i \leq 2$ (Recall $\mathcal{E}_i \define \mathcal{Y}_i ( \mathcal{Y}^{H}_i \mathcal{Y}_i )^{\dagger}$) and   
$\left\Vert \delta \mathcal{D} \right\Vert \leq \epsilon_{D} \left\Vert \mathcal{D} \right\Vert$, then the Eq.~\eqref{eq: pert proof 4} can be further reduced as:
\begin{eqnarray}
\left\Vert \mathcal{Y} - \mathcal{X} \right\Vert &=&  (1 + \epsilon_D) \left\Vert \mathcal{D} \right\Vert (2 \epsilon^2_A \left\Vert \mathcal{A} \right\Vert^2  \left\Vert \mathcal{A}^{\dagger} \right\Vert  + \epsilon_A^{3} \left\Vert \mathcal{A} \right\Vert^3 +
\epsilon_A^{4} \left\Vert \mathcal{A} \right\Vert^4 \left\Vert \mathcal{A}^{\dagger} \right\Vert ) + \nonumber \\
&  & \epsilon_D \left\Vert \mathcal{A}^{\dagger} \right\Vert \left\Vert \mathcal{D} \right\Vert, 
\end{eqnarray}
and since $\left\Vert \mathcal{D} \right\Vert = \left\Vert \mathcal{A} \mathcal{X} \right\Vert \leq \left\Vert \mathcal{A} \right\Vert \left\Vert \mathcal{X} \right\Vert$, we have 
\begin{eqnarray}
\frac{\left\Vert \mathcal{Y} - \mathcal{X} \right\Vert}{\left\Vert \mathcal{X} \right\Vert} &\leq& (1 + \epsilon_D) \left\Vert \mathcal{A} \right\Vert^3 (2 \epsilon^2_{A} \left\Vert \mathcal{A}^{\dagger} \right\Vert + \epsilon^3_{A} \left\Vert \mathcal{A} \right\Vert + \epsilon^4_{A} \left\Vert \mathcal{A} \right\Vert^2 \left\Vert \mathcal{A}^{\dagger} \right\Vert) +  \nonumber \\
&  &\epsilon_D \left\Vert \mathcal{A} \right\Vert \left\Vert \mathcal{A}^{\dagger} \right\Vert 
\end{eqnarray}

The theorem is proved. 

$\hfill \Box$

\subsection{Numerical Evaluation}\label{sec:Numerical Evaluation}


In this section, we will apply normalized error bound results derived in Section~\ref{sec:Sensitivity Analysis} to the multilinear equation $\mathcal{A}\mathcal{X} = \mathcal{D}$ with following tensors:
\begin{eqnarray}\label{eq:num tensor A}
\mathcal{A}&=& \left[
    \begin{array}{cc : cc}
       1 & -1 & 0 & 0  \\
       0 & 0 & -1 & 0 \\ \hdashline[2pt/2pt]
       0 &  1 & 0 & 0  \\
       0 &  0 & 1 & 0 
    \end{array}
\right],
\end{eqnarray}
\begin{eqnarray}\label{eq:num tensor A dagger}
\mathcal{A}^{\dagger}&=& \left[
    \begin{array}{cc : cc}
       1 & 0 & 0 & 0  \\
       1 & 0 & 1 & 0 \\ \hdashline[2pt/2pt]
       0 &  -0.5 & 0 & 0  \\
       0 &  0.5 & 0 & 0 
    \end{array}
\right],
\end{eqnarray}
and
\begin{eqnarray}\label{eq:num tensor D}
\mathcal{D}&=& \left[
    \begin{array}{cc }
       1 & 2 \\
       1 & 1 
    \end{array}
\right].
\end{eqnarray}

In Fig.~\ref{AX_D_Moore_Sensi_epsA_09_05_01}, the normalized error bound for the multilinear system $\mathcal{A}\mathcal{X} = \mathcal{D}$ is presented against with the change of the Frobenius norm of the tesnor $\mathcal{A}$ according to the Theorem~\ref{Theorem: sensitivity of AX equal D}. Fig.~\ref{AX_D_Moore_Sensi_epsA_09_05_01} delineates the normalized error bound with respect to three different perturbation values $\epsilon_A=0.09, 0.05, 0.01$ of the tensor $\mathcal{A}$ subject to the perturbation value $\epsilon_D = 0.01$ of the tensor $\mathcal{D}$. The way we change the Frobenius norm of the tesnor $\mathcal{A}$ is by scaling the tensor $\mathcal{A}$ with some positive number $\alpha$, i.e., $\alpha \mathcal{A}$ is a tensor obtained by multiplying the value $\alpha$ to each entries of the tensor $\mathcal{A}$. We observe that the normalization error $E_n$ increases with the increase of the perturbation value $\epsilon_A$. Given the same perturbaiton value $\epsilon_A$,  the normalized error bound $E_n$ can achieve its minimum by scaling the tesnor $\mathcal{A}$ properly. For example, when the value $\epsilon_A$ is $0.09$, the minimum error bound happens when the value of $\left\Vert \mathcal{A} \right\Vert$ is about $2.5$.

\begin{figure}[htbp]
	\centerline{\includegraphics[width=\columnwidth,draft=false]
		{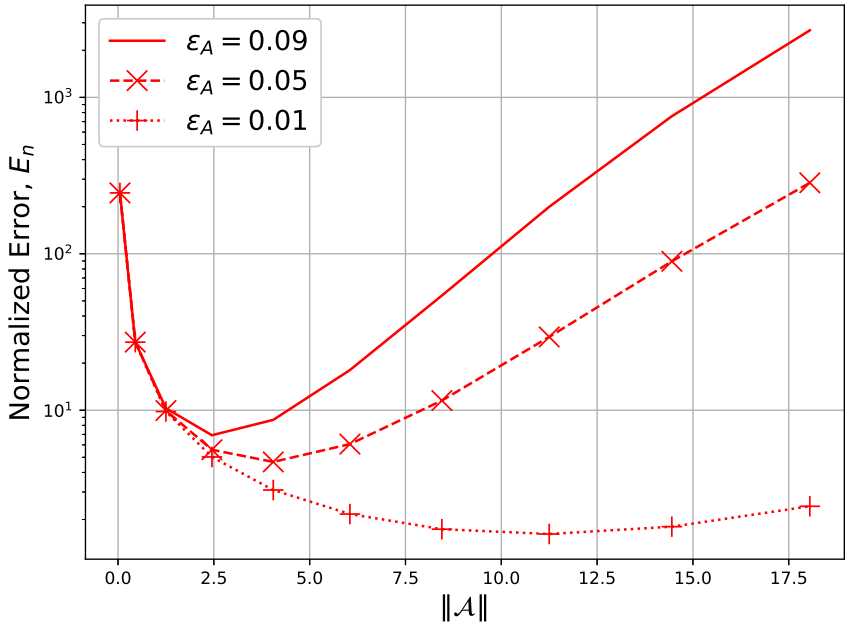}}
	\caption{The normalized error bound $E_n$ for the perturbed multilinear system $\mathcal{A}\mathcal{X} = \mathcal{D}$ with respect to the tesnor norm $  \left\Vert \mathcal{A} \right\Vert  $ for 
different $\epsilon_A$ values when the tesnor norm $\left\Vert \mathcal{D} \right\Vert$ is $7$ and $\epsilon_D = 0.01$. }\label{AX_D_Moore_Sensi_epsA_09_05_01}
\end{figure}

In Fig.~\ref{AX_D_Moore_Sensi_epsD_09_05_01}, the normalized error bound for the multilinear system $\mathcal{A}\mathcal{X} = \mathcal{D}$ is presented against with the change of the Frobenius norm of the tesnor $\mathcal{A}$. Fig.~\ref{AX_D_Moore_Sensi_epsD_09_05_01} plots the normalized error bound with respect to three different perturbation values $\epsilon_D=0.09, 0.05, 0.01$ of the tensor $\mathcal{D}$ subject to the perturbation value $\epsilon_A = 0.01$ of the tensor $\mathcal{A}$. We find that the normalization error $E_n$ increases with the increase of the perturbation value $\epsilon_D$. Given the same perturbaiton value $\epsilon_A$,  the bound $E_n$ also can achieve its minimum by scaling the tesnor $\mathcal{A}$ properly. For example, when the value $\epsilon_D$ is $0.01$, the minimum error bound happens when the value of $\left\Vert \mathcal{A} \right\Vert$ is about $1.25$. Compared to Fig.~\ref{AX_D_Moore_Sensi_epsD_09_05_01}, the error bounds difference between various perturbation values $\epsilon_A$ becomes more significant when the value of the Frobenius norm of the tensor $\mathcal{A}$ increases. On the other hand, the error bounds difference between various perturbation values $\epsilon_D$ becomes less significant when the value of the Frobenius norm of the tensor $\mathcal{A}$ increases. Both figures show that the error bound variation is more sensitive with respect to the Frobenius norm of the tensor $\mathcal{A}$ for smaller value range of $\left\Vert \mathcal{A} \right\Vert$.

\begin{figure}[htbp]
	\centerline{\includegraphics[width=\columnwidth,draft=false]
		{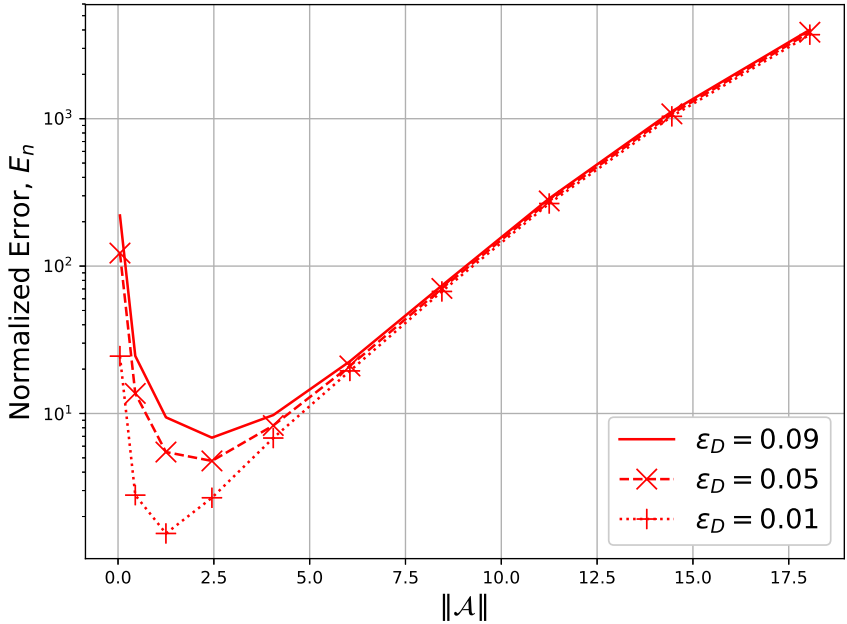}}
	\caption{The normalized error $E_n$ for the perturbed multilinear system equations $\mathcal{A}\mathcal{X} = \mathcal{D}$ with respect to the tesnor norm $  \left\Vert \mathcal{A} \right\Vert  $ for 
different $\epsilon_D$ when the tesnor norm $  \left\Vert \mathcal{D} \right\Vert  $ is $7$.}\label{AX_D_Moore_Sensi_epsD_09_05_01}
\end{figure}

\section{Conclusions}\label{sec:Conclusions}

Motivated by great applications of the Sherman–Morrison–\\
Woodbury matrix identity, analogously, we developed the Sherman–Morrison–\\
Woodbury identity for tensors to facilitate the tensor inversion computation with those benefits in the matrix inversion computation when the correction of the original tensors is required. We first established the Sherman–Morrison–Woodbury identity for invertible tensors. Furthermore, we generalized the Sherman–Morrison–Woodbury identity for tensor with Moore-Penrose inverse by using orthogonal projection of the correction tensor part into the original tensor and its Hermitian tensor. Finally, we applied the Sherman–Morrison–Woodbury identity to characterize the error bound for the solution of a multilinear system between the original system and the corrected system, i.e., the coefficient tensors are corrected by other tensors with same dimensions.

There are several possible future works that can be extended based on current work. Because we can quantify the normalized error bound with respect to perturbation values and the Frobenius norm of the coefficient tensor, the next question is how to design a robust multilinear system to have the minimum normalized solution error given perturbation values. Such robust design should be crucial in many engineering problems which are modeled by multilinear systems. We have to decompose the perturbed tensor in the Eq.~\eqref{delta A decomp} in order to apply our result, similar to the matrix case, how can we select low rank decomposition for the perturbed tensor is the second direction for the future research. Since we have developed a new Sherman–Morrison–Woodbury identity for tensor, it will be interested in finding more impactful applications based on this new identity. We expect this new identity will shed light on the development of more efficient tensor-based calculations in the near future.

\section*{Acknowledgments}
The helpful comments of the referees are gratefully acknowledged.

\bibliographystyle{siamplain}
\bibliography{TensorSheMorWoodbury_Bib}

\begin{thebibliography}{10}

\bibitem{balan2010applications}
{\sc V.~Balan and N.~Perminov}, {\em Applications of resultants in the spectral
  m-root framework}, Applied Sciences, 12 (2010), pp.~20--29.

\bibitem{batselier2017tensor}
{\sc K.~Batselier, Z.~Chen, and N.~Wong}, {\em A tensor network kalman filter
  with an application in recursive mimo volterra system identification},
  Automatica, 84 (2017), pp.~17--25.

\bibitem{behera2017further}
{\sc R.~Behera and D.~Mishra}, {\em Further results on generalized inverses of
  tensors via the einstein product}, Linear and Multilinear Algebra, 65 (2017),
  pp.~1662--1682.

\bibitem{brazell2013solving}
{\sc M.~Brazell, N.~Li, C.~Navasca, and C.~Tamon}, {\em Solving multilinear
  systems via tensor inversion}, SIAM Journal on Matrix Analysis and
  Applications, 34 (2013), pp.~542--570.

\bibitem{brinton2016mining}
{\sc C.~G. Brinton, S.~Buccapatnam, M.~Chiang, and H.~V. Poor}, {\em Mining
  mooc clickstreams: Video-watching behavior vs. in-video quiz performance},
  IEEE Transactions on Signal Processing, 64 (2016), pp.~3677--3692.

\bibitem{bu2014inverse}
{\sc C.~Bu, X.~Zhang, J.~Zhou, W.~Wang, and Y.~Wei}, {\em The inverse, rank and
  product of tensors}, Linear Algebra and Its Applications, 446 (2014),
  pp.~269--280.

\bibitem{cai2006tensor}
{\sc D.~Cai, X.~He, and J.~Han}, {\em Tensor space model for document
  analysis}, in Proceedings of the 29th annual international ACM SIGIR
  conference on Research and development in information retrieval, Aug. 2006,
  pp.~625--626.

\bibitem{cardoso1999high}
{\sc J.-F. Cardoso}, {\em High-order contrasts for independent component
  analysis}, Neural computation, 11 (1999), pp.~157--192.

\bibitem{cartwright2013number}
{\sc D.~Cartwright and B.~Sturmfels}, {\em The number of eigenvalues of a
  tensor}, Linear algebra and its applications, 438 (2013), pp.~942--952.

\bibitem{chen2019incremental}
{\sc D.~Chen, Y.~Tang, H.~Zhang, L.~Wang, and X.~Li}, {\em Incremental
  factorization of big time series data with blind factor approximation}, IEEE
  Transactions on Knowledge and Data Engineering,  (2019).

\bibitem{cheng2020learning}
{\sc L.~Cheng, X.~Tong, S.~Wang, Y.-C. Wu, and H.~V. Poor}, {\em Learning
  nonnegative factors from tensor data: Probabilistic modeling and inference
  algorithm}, IEEE Transactions on Signal Processing, 68 (2020),
  pp.~1792--1806.

\bibitem{cheng2018scaling}
{\sc L.~Cheng, Y.-C. Wu, and H.~V. Poor}, {\em Scaling probabilistic tensor
  canonical polyadic decomposition to massive data}, IEEE Transactions on
  Signal Processing, 66 (2018), pp.~5534--5548.

\bibitem{cooper2020adjacency}
{\sc J.~Cooper}, {\em Adjacency spectra of random and complete hypergraphs},
  Linear Algebra and its Applications,  (2020).

\bibitem{cui2019preconditioned}
{\sc L.-B. Cui, M.-H. Li, and Y.~Song}, {\em Preconditioned tensor splitting
  iterations method for solving multi-linear systems}, Applied Mathematics
  Letters, 96 (2019), pp.~89--94.

\bibitem{de2007fourth}
{\sc L.~De~Lathauwer, J.~Castaing, and J.-F. Cardoso}, {\em Fourth-order
  cumulant-based blind identification of underdetermined mixtures}, IEEE
  Transactions on Signal Processing, 55 (2007), pp.~2965--2973.

\bibitem{deif2012sensitivity}
{\sc A.~Deif}, {\em Sensitivity analysis in linear systems}, Springer Science
  \& Business Media, 2012.

\bibitem{dhillon2008matrix}
{\sc I.~S. Dhillon and J.~A. Tropp}, {\em Matrix nearness problems with bregman
  divergences}, SIAM Journal on Matrix Analysis and Applications, 29 (2008),
  pp.~1120--1146.

\bibitem{ding2018fast}
{\sc W.~Ding, M.~Ng, and Y.~Wei}, {\em Fast computation of stationary joint
  probability distribution of sparse markov chains}, Applied Numerical
  Mathematics, 125 (2018), pp.~68--85.

\bibitem{ding2016solving}
{\sc W.~Ding and Y.~Wei}, {\em Solving multi-linear systems with
  $\mathcal{M}$-tensors}, Journal of Scientific Computing, 68 (2016),
  pp.~689--715.

\bibitem{guglielmi2015low}
{\sc N.~Guglielmi, D.~Kressner, and C.~Lubich}, {\em Low rank differential
  equations for hamiltonian matrix nearness problems}, Numerische Mathematik,
  129 (2015), pp.~279--319.

\bibitem{hu2015laplacian}
{\sc S.~Hu and L.~Qi}, {\em The laplacian of a uniform hypergraph}, Journal of
  Combinatorial Optimization, 29 (2015), pp.~331--366.

\bibitem{hu2012geometric}
{\sc S.~Hu, L.~Qi, and G.~Zhang}, {\em The geometric measure of entanglement of
  pure states with nonnegative amplitudes and the spectral theory of
  nonnegative tensors}, arXiv preprint arXiv:1203.3675,  (2012).

\bibitem{ji2018drazin}
{\sc J.~Ji and Y.~Wei}, {\em The drazin inverse of an even-order tensor and its
  application to singular tensor equations}, Computers \& Mathematics with
  Applications, 75 (2018), pp.~3402--3413.

\bibitem{jin2017generalized}
{\sc H.~Jin, M.~Bai, J.~Ben{\'\i}tez, and X.~Liu}, {\em The generalized
  inverses of tensors and an application to linear models}, Computers \&
  Mathematics with Applications, 74 (2017), pp.~385--397.

\bibitem{kanatsoulis2019regular}
{\sc C.~I. Kanatsoulis, N.~D. Sidiropoulos, M.~Ak{\c{c}}akaya, and X.~Fu}, {\em
  Regular sampling of tensor signals: Theory and application to fmri}, in
  ICASSP 2019-2019 IEEE International Conference on Acoustics, Speech and
  Signal Processing (ICASSP), IEEE, 2019, pp.~2932--2936.

\bibitem{kolda2006tophits}
{\sc T.~Kolda and B.~Bader}, {\em The tophits model for higher-order web link
  analysis}, in Workshop on link analysis, counterterrorism and security,
  vol.~7, Apr. 2006, pp.~26--29.

\bibitem{kolda2009tensor}
{\sc T.~G. Kolda and B.~W. Bader}, {\em Tensor decompositions and
  applications}, SIAM review, 51 (2009), pp.~455--500.

\bibitem{kolda2005higher}
{\sc T.~G. Kolda, B.~W. Bader, and J.~P. Kenny}, {\em Higher-order web link
  analysis using multilinear algebra}, in Proceedings of IEEE International
  Conference on Data Mining (ICDM), November 2005, pp.~8--15.

\bibitem{Kwak2015High}
{\sc J.~Kwak and C.-H. Lee}, {\em A high-order markov-chain-based scheduling
  algorithm for low delay in csma networks}, IEEE/ACM Transactions on
  Networking, 24 (2015), pp.~2278--2290.

\bibitem{lai2009introduction}
{\sc W.~M. Lai, D.~H. Rubin, E.~Krempl, and D.~Rubin}, {\em Introduction to
  continuum mechanics}, Butterworth-Heinemann, 2009.

\bibitem{li2012characteristic}
{\sc A.-M. Li, L.~Qi, and B.~Zhang}, {\em E-characteristic polynomials of
  tensors}, arXiv preprint arXiv:1208.1607,  (2012).

\bibitem{li2017splitting}
{\sc D.-H. Li, S.~Xie, and H.-R. Xu}, {\em Splitting methods for tensor
  equations}, Numerical Linear Algebra with Applications, 24 (2017), p.~e2102.

\bibitem{li2011general}
{\sc Q.~Li, X.~Shi, and D.~Schonfeld}, {\em A general framework for robust
  hosvd-based indexing and retrieval with high-order tensor data}, in 2011 IEEE
  International Conference on Acoustics, Speech and Signal Processing (ICASSP),
  IEEE, May 2011, pp.~873--876.

\bibitem{li2019c}
{\sc W.~Li, R.~Ke, W.-K. Ching, and M.~K. Ng}, {\em A c-eigenvalue problem for
  tensors with applications to higher-order multivariate markov chains},
  Computers \& Mathematics with Applications, 78 (2019), pp.~1008--1025.

\bibitem{li2012har}
{\sc X.~Li, M.~K. Ng, and Y.~Ye}, {\em Har: hub, authority and relevance scores
  in multi-relational data for query search}, in Proceedings of the 2012 SIAM
  International Conference on Data Mining, SIAM, Apr. 2012, pp.~141--152.

\bibitem{liang2019further}
{\sc M.~Liang and B.~Zheng}, {\em Further results on moore--penrose inverses of
  tensors with application to tensor nearness problems}, Computers \&
  Mathematics with Applications, 77 (2019), pp.~1282--1293.

\bibitem{liu2019relaxation}
{\sc D.~Liu, W.~Li, and S.-W. Vong}, {\em Relaxation methods for solving the
  tensor equation arising from the higher-order markov chains}, Numerical
  Linear Algebra with Applications, 26 (2019), p.~e2260.

\bibitem{liu2005text}
{\sc N.~Liu, B.~Zhang, J.~Yan, Z.~Chen, W.~Liu, F.~Bai, and L.~Chien}, {\em
  Text representation: From vector to tensor}, in Fifth IEEE International
  Conference on Data Mining (ICDM'05), IEEE, Nov. 2005, pp.~4--pp.

\bibitem{OVMMUPH2011}
{\sc O.~V. Morozov, M.~Unser, and P.~Hunziker}.

\bibitem{ng2011multirank}
{\sc M.~K.-P. Ng, X.~Li, and Y.~Ye}, {\em Multirank: co-ranking for objects and
  relations in multi-relational data}, in Proceedings of the 17th ACM SIGKDD
  international conference on Knowledge discovery and data mining, Aug. 2011,
  pp.~1217--1225.

\bibitem{ortiz2019sparse}
{\sc G.~Ortiz-Jim{\'e}nez, M.~Coutino, S.~P. Chepuri, and G.~Leus}, {\em Sparse
  sampling for inverse problems with tensors}, IEEE Transactions on Signal
  Processing, 67 (2019), pp.~3272--3286.

\bibitem{phan2018error}
{\sc A.-H. Phan, P.~Tichavsk{\`y}, and A.~Cichocki}, {\em Error preserving
  correction: A method for cp decomposition at a target error bound}, IEEE
  Transactions on Signal Processing, 67 (2018), pp.~1175--1190.

\bibitem{precup2012novel}
{\sc R.-E. Precup, C.-A. Dragos, S.~Preitl, M.-B. Radac, and E.~M. Petriu},
  {\em Novel tensor product models for automatic transmission system control},
  IEEE Systems Journal, 6 (2012), pp.~488--498.

\bibitem{qi2012minimum}
{\sc L.~Qi}, {\em The minimum hartree value for the quantum entanglement
  problem}, arXiv preprint arXiv:1202.2983,  (2012).

\bibitem{qi2018tensor}
{\sc L.~Qi, H.~Chen, and Y.~Chen}, {\em Tensor eigenvalues and their
  applications}, vol.~39, Springer, 2018.

\bibitem{qi2017tensor}
{\sc L.~Qi and Z.~Luo}, {\em Tensor analysis: spectral theory and special
  tensors}, SIAM, 2017.

\bibitem{qi2008d}
{\sc L.~Qi, Y.~Wang, and E.~X. Wu}, {\em D-eigenvalues of diffusion kurtosis
  tensors}, Journal of Computational and Applied Mathematics, 221 (2008),
  pp.~150--157.

\bibitem{qi2010higher}
{\sc L.~Qi, G.~Yu, and E.~X. Wu}, {\em Higher order positive semidefinite
  diffusion tensor imaging}, SIAM Journal on Imaging Sciences, 3 (2010),
  pp.~416--433.

\bibitem{saad2003iterative}
{\sc Y.~Saad}, {\em Iterative methods for sparse linear systems}, vol.~82,
  siam, 2003.

\bibitem{sahoo2020reverse}
{\sc J.~K. Sahoo and R.~Behera}, {\em Reverse-order law for core inverse of
  tensors}, Computational and Applied Mathematics, 39 (2020), pp.~1--22.

\bibitem{sawilla2008survey}
{\sc R.~Sawilla}, {\em A survey of data mining of graphs using spectral graph
  theory}, Defence R \& D Canada-Ottawa, 2008.

\bibitem{shin2016fully}
{\sc K.~Shin, L.~Sael, and U.~Kang}, {\em Fully scalable methods for
  distributed tensor factorization}, IEEE Transactions on Knowledge and Data
  Engineering, 29 (2016), pp.~100--113.

\bibitem{sun2016moore}
{\sc L.~Sun, B.~Zheng, C.~Bu, and Y.~Wei}, {\em Moore--penrose inverse of
  tensors via einstein product}, Linear and Multilinear Algebra, 64 (2016),
  pp.~686--698.

\bibitem{sun2018generalized}
{\sc L.~Sun, B.~Zheng, Y.~Wei, and C.~Bu}, {\em Generalized inverses of tensors
  via a general product of tensors}, Frontiers of Mathematics in China, 13
  (2018), pp.~893--911.

\bibitem{tang2012cross}
{\sc J.~Tang, G.-J. Qi, L.~Zhang, and C.~Xu}, {\em Cross-space affinity
  learning with its application to movie recommendation}, IEEE Transactions on
  Knowledge and Data Engineering, 25 (2012), pp.~1510--1519.

\bibitem{tao2007general}
{\sc D.~Tao, X.~Li, X.~Wu, and S.~J. Maybank}, {\em General tensor discriminant
  analysis and gabor features for gait recognition}, IEEE transactions on
  pattern analysis and machine intelligence, 29 (2007), pp.~1700--1715.

\bibitem{tucker1966some}
{\sc L.~R. Tucker}, {\em Some mathematical notes on three-mode factor
  analysis}, Psychometrika, 31 (1966), pp.~279--311.

\bibitem{wei2016theory}
{\sc Y.~Wei and W.~Ding}, {\em Theory and computation of tensors:
  multi-dimensional arrays}, Academic Press, 2016.

\bibitem{xie2018tensor}
{\sc Z.-J. Xie, X.-Q. Jin, and Y.-M. Wei}, {\em Tensor ethods for solving
  symmetric $m$-tensor systems}, Journal of Scientific Computing, 74 (2018),
  pp.~412--425.

\bibitem{zhang2013gradient}
{\sc F.~Zhang, B.~Zhou, and L.~Peng}, {\em Gradient skewness tensors and local
  illumination detection for images}, Journal of Computational and Applied
  Mathematics, 237 (2013), pp.~663--671.

\bibitem{zhou2019tensor}
{\sc M.~Zhou, Y.~Liu, Z.~Long, L.~Chen, and C.~Zhu}, {\em Tensor rank learning
  in cp decomposition via convolutional neural network}, Signal Processing:
  Image Communication, 73 (2019), pp.~12--21.

\end{thebibliography}
\end{document}